\documentclass[ 11pt]{article}

\usepackage[utf8]{inputenc}  
\usepackage[T1]{fontenc}      
\usepackage{geometry}      
\usepackage[english]{babel} 
\usepackage{cite}
\usepackage{amssymb,  xfrac, stmaryrd, blkarray, multirow, enumerate,
  verbatim,xifthen}   
\usepackage{float}
\usepackage{amsmath} 
\usepackage{amscd}
\usepackage{graphicx, color, pinlabel}    
\usepackage{amsthm}                 
\usepackage{latexsym}   
\usepackage[all]{xy}
\usepackage{hyperref}
\usepackage{mathtools}
\usepackage{fancyhdr}
\newtheorem{thm}{Theorem}[section]

\theoremstyle{definition}
\newtheorem{definition}[thm]{Definition}
\newtheorem{deflem}[thm]{Lemma-Definition}
\newtheorem{defcor}[thm]{Corollary-Definition}
\newtheorem{rmk}[thm]{Remark}
\newtheorem*{thm*}{Theorem}
\newtheorem*{thx}{Acknowledgements}

\newtheorem*{definition*}{Definition}
\newtheorem*{question}{Question}
\newtheorem*{cor*}{Corollary}
\newtheorem{prop}[thm]{Proposition}
\newtheorem{cor}[thm]{Corollary}
\newtheorem*{prop*}{Proposition}
\newtheorem{lem}[thm]{Lemma}
\newtheorem*{thmA}{Theorem A}
\newtheorem*{thmB}{Theorem B}
\newtheorem*{corC}{Corollary C}
\newtheorem*{thmE}{Theorem E}
\newtheorem*{corD}{Corollary D}
\newtheorem*{corF}{Corollary F}
\newtheorem*{corG}{Corollary G}

\def\lquotient#1#2{%
\makeatletter
\lower.6ex\hbox{$#1$}\backslash\raise.3ex\hbox{$#2$}%
\makeatother
}

\def\rquotient#1#2{%
\makeatletter
\raise.6ex\hbox{$#1$}/\lower.2ex\hbox{$#2$}%
\makeatother
}

\newcommand{\bbR}{{\mathbb R}}

\newcommand{\bbZ}{{\mathbb Z}}

\newcommand{\ra}{\rightarrow}

\newcommand{\hra}{\hookrightarrow}

\def\rquotient#1#2{%
\makeatletter
\raise.3ex\hbox{$#1$}/\lower.3ex\hbox{$#2$}%
\makeatother
}	

\makeatletter
\newcommand{\subjclass}[2][2010]{%
  \let\@oldtitle\@title%
  \gdef\@title{\@oldtitle\footnotetext{#1 \emph{Mathematics subject classification.} #2}}%
}
\newcommand{\keywords}[1]{%
  \let\@@oldtitle\@title%
  \gdef\@title{\@@oldtitle\footnotetext{\emph{Key words and phrases.} #1.}}%
}
\makeatother

\newcommand{\Address}{{
  \bigskip
  \small

   \textsc{Department of Mathematics, University of Vienna,
Oskar-Morgenstern-Platz 1, 1090 Vienna, Austria.}\par\nopagebreak
  \textit{E-mail address}: \texttt{alexandre.martin@univie.ac.at}

}}

\makeatletter
\newdimen\bibindent
\makeatother

\title{\textbf{Acylindrical actions on CAT(0) square complexes}}  
\author{Alexandre Martin}
\date{}


\begin{document}
\maketitle

\begin{abstract} 
For group actions on hyperbolic CAT(0) square complexes, we show that the acylindricity of the action is equivalent to a weaker form of acylindricity phrased purely in terms of stabilisers of points, which  has the advantage of being much more tractable for actions on non-locally compact spaces.  For group actions on general CAT(0) square complexes, we show that an analogous characterisation holds for the so-called WPD condition.
As an application, we study the geometry of generalised Higman groups on at least $5$ generators, the first historical examples of finitely presented infinite groups without non-trivial finite quotients. We show that these groups act acylindrically on the CAT(-1) polygonal complex naturally associated to their presentation. As a consequence, such groups satisfy a strong version of the Tits alternative and are residually $F_2$-free, that is, every element of the group survives in a quotient that does not contain a non-abelian free subgroup.

\end{abstract}

\medskip

Acylindrical actions were first considered by Sela for groups acting on simplicial trees \cite{SelaAcylindrical}. In Sela's terminology, given a (minimal) action of a group on a simplicial tree, the action is said to be acylindrical if there exists an integer $k \geq 1$ such that no non-trivial element of the group fixes pointwise two points at distance at least $k$.   This definition was extended to actions on arbitrary geodesic metric spaces by Bowditch \cite{BowditchTightGeodesics}, in his study of the action of the mapping class group of a closed hyperbolic surface on its associated curve complex. Recall that an action of a group $G$ on a metric space $X$ is \textit{acylindrical} if for every $r \geq 0$ there exist constants $L(r), N(r) \geq 0$ such that for every points $x, y$ of $X$ at distance at least $L(r)$, there are at most $N(r)$ elements $h$ of $G$ such that $d(x, hx), d(y,hy) \leq r$. For $r=0$, one recovers Sela's definition of acylindricity, at least for torsion-free groups. As noticed by Bowditch \cite{BowditchTightGeodesics}, in the case of group actions on simplicial trees,  acylindricity is equivalent to this weaker acylindricity condition at $r=0$. 

Groups that act acylindrically on a hyperbolic spaces share many features with relatively hyperbolic groups, and techniques from dynamics in negative curvature are available to study them. Unfortunately, acylindrical actions on hyperbolic spaces seldom appear naturally in geometric group theory. However,  weaker but more frequent forms of acylindricity are sufficient to apply this circle of ideas, as noticed by various authors. In a nutshell, the general philosophy is the following: Given a group $G$ acting on a geodesic metric space $X$, if the space $X$ `looks hyperbolic in the direction of some loxodromic element $g$' and if $G$ acts on $X$ `acylindrically in the direction of $g$', then most of the machinery carries over to this situation. This was done for instance in \cite{BestvinaFujiwaraMappingClassGroups, SistoContractingElements, GruberSistoAcylindricallyHyperbolic, MinasyanOsinTrees, CapraceHumeAcylindrical, CantatLamyCremonaNormalSubgroups}. Osin showed the equivalence of all the weak notions of acylindricity involved \cite{OsinAcylindricallyHyperbolic} and proved them to be equivalent to the existence of a non-elementary action of the group on some (generally quite complicated) hyperbolic space. Groups satisfying those properties are now referred to as \textit{acylindrically hyperbolic} groups. 

To date, the most general criterion to show the acylindrical hyperbolicity of a group is the following: 

\begin{thm*}[see {\cite[Theorem H]{BestvinaBrombergFujiwara}}]
 Let $G$ be a group acting by isometries on a geodesic metric space $X$. Let $g$ be an infinite order element with quasi-isometrically embedded orbits  and assume that the following holds: 
 \begin{itemize}
  \item $g$ is a \textit{strongly contracting element}, that is, there exists a point $x$ of $X$ such that the closest-point projections on the orbit $\langle g \rangle x$ of the balls of $X$ that are disjoint from  $\langle g \rangle x$ have uniformly bounded diameter,
  \item $g$ satisfies the \textit{WPD condition}, that is, for every $r \geq 0$ there exist constants $m(r), N(r) \geq 0$ such that for every point $x$ of $X$, there are at most $N(r)$ elements $h$ of $G$ such that $d(x, hx), d(g^{m(r)}x,hg^{m(r)}x) \leq r$.
 \end{itemize}
Then $G$ is either virtually cyclic or acylindrically hyperbolic.
\end{thm*}

However, checking the acylindricity of an action, or checking that a given hyperbolic element satisfies the WPD condition, is generally tedious. Indeed, while such conditions for actions on simplicial trees boil down to conditions about stabilisers of pairs of points, these conditions involve, for general actions, `coarse' stabilisers of pairs of points, that is, group elements moving a pair of points by a uniformly bounded amount. And while stabilisers of pairs of points lead us to study the set of geodesics between two given points, coarse stabilisers of pairs of points  requires a fine understanding of the set of geodesics \textit{ between metric balls}. In the case of actions on non-locally compact spaces of dimension at least $2$, understanding such sets of geodesics is a highly non-trivial problem.

Let us illustrate these complications with the case of the mapping class group of a closed hyperbolic surface acting on its curve complex. Two vertices of the curve complex at distance at least $3$ correspond to curves that fill the surface. Thus, the subgroup fixing two such vertices is finite and, what is more, of uniformly bounded cardinality by a theorem of Hurwitz and the Nielsen Realisation Theorem.  In particular, such an action is weakly acylindrical. By contrast, the original proof of the acylindricity of the action by Bowditch required the introduction of the so-called \textit{tight geodesics} of the curve complex and a fine understanding of their combinatorial geometry \cite{BowditchTightGeodesics}. \\

When trying to prove the acylindrical hyperbolicity of a group through its action on a given geodesic metric space, a natural approach is to identify classes of group elements, defined in terms of the geometry of the space acted upon and the dynamics of the action at hand, which are good candidates for being strongly contracting and satisfying the WPD condition. 

In this article however, we are interested in a slightly different, but related question. We want to find weaker, but more tractable, notions of acylindricity and of the WPD condition that make sense for general actions on geodesic metric spaces, which would nonetheless be equivalent to the usual notions for actions on sufficiently well behaved classes of spaces, paralleling the aforementioned situation for actions on simplicial trees. Here, `tractable' should be interpreted as `phrased in terms of stabilisers of points (as in Sela's original definition) rather than in terms of coarse stabilisers of points'. Let us make this question more precise:

\begin{question}\begin{itemize}
\item Is there a weaker  form of acylindricity for groups acting on geodesic metric spaces, phrased purely in terms of stabilisers of points, such that, for some nice classes of spaces, such a weaker form of acylindricity is equivalent to the acylindricity of the action? 

\item Is there a weaker  form of acylindricity for groups acting on geodesic metric spaces, phrased purely in terms of stabilisers of points, such that for such actions on some nice classes of spaces, \textit{every} strongly contracting element $g$ automatically satisfies the WPD condition?  

\item Is there a weaker form of the WPD condition for groups acting on geodesic metric spaces, phrased purely in terms of stabilisers of points, such that, for some nice classes of spaces, \textit{every} strongly contracting element $g$ that satisfies such a weak form of the WPD condition also satisfies the WPD condition? 
\end{itemize}
\end{question}

Weak forms of acylindricity have been considered by various authors, for instance Delzant \cite{DelzantMonodromies} and Hamenst\"adt \cite{HamenstadtWeakAcylindricity}. In this article, we will be interested in the following weak form of acylindricity, which is the case $r=0$ mentioned earlier in the usual definition of acylindricity: 

\begin{definition*}[weak acylindricity]
 Let $G$ be a group acting on a geodesic metric space $X$. We say that the action is \textit{weakly acylindrical} if there exist constants $L, N \geq 0$ such that two points of $X$ at distance at least $L$ are pointwise fixed by at most $N$ elements of $G$. 
\end{definition*}

It should be noted that for groups with a uniform bound on the size of their finite subgroups, weak acylindricity amounts to requiring that two points sufficiently (and uniformly) far apart are fixed by only finitely many elements.

Such a weaker notion of acylindricity is particularly useful when dealing with group actions on polyhedral complexes for which we have a good understanding of the inclusions of face stabilisers. For instance, an amalgamated product of torsion-free groups of the form $A \underset{C}{*} B$, where $B$ embeds malnormally in $A$, acts weakly acylindrically on its Bass--Serre tree (this example appeared in \cite{SelaAcylindrical}). 
More generally, weak acylindricity is  well suited to study groups presented as fundamental groups of complexes of groups, as we will illustrate below. \\

While acylindricity is a priori much stronger than the weak acylindricity considered here, the main goal of this paper is to show that these notions are in fact equivalent when dealing with actions on particularly well behaved spaces, generalising Bowditch's observation for actions on simplicial trees. Our first result is the following:

\begin{thmA}
 Let $G$ be a group acting weakly acylindrically on a hyperbolic CAT(0) square complex. Then the action is acylindrical. 
\end{thmA}

As an application, we consider \textit{generalised Higman groups}, as introduced by Higman \cite{HigmanGroup}. The generalised Higman groups $H_n, n \geq 4$, were defined by Higman by generators and relations: 
$$H_n:= \langle a_i, i \in \bbZ / n \bbZ ~|~ a_ia_{i+1}a_i^{-1}= a_{i+1}^{2}, i \in \bbZ / n \bbZ  \rangle,$$
These groups are historically the first examples of finitely presented infinite groups without non-trivial finite quotients \cite{HigmanGroup}. In \cite{MartinHigmanCubical}, the author studied the action of $H_4$ on a (non-hyperbolic) CAT(0) square complex naturally associated to its presentation. For $n \geq 5$, $H_n$ acts cocompactly on a CAT(-1) polygonal complex $X_n$, and in particular on a hyperbolic CAT(0) square complex by taking an appropriate subdivision. Such an action is easily shown to be weakly acylindrical. In particular, we obtain the following result: 

\begin{thmB}
 For $n \geq 5$, the action of $H_n$ on $X_n$ is acylindrical.
\end{thmB}

The acylindrical hyperbolicity of generalised Higman groups was first proved by Minasyan--Osin \cite{MinasyanOsinTrees}, using prior work of Schupp \cite{SchuppHigmanSQ}. While acylindrical hyperbolicity alone implies strong consequences for the group (see \cite{OsinAcylindricallyHyperbolic} and details therein), having this well understood acylindrical action of $H_n$ on a hyperbolic complex allows us to obtain results which do not follow solely from the abstract acylindrical hyperbolicity of the group. In particular, we obtain the following: 

\begin{corC}[Strong Tits alternative for generalised Higman groups] For $n \geq 5$, a non-cyclic subgroup of $H_n$ is either contained in a vertex stabiliser, hence embeds in $BS(1,2)$, or is acylindrically hyperbolic.
\end{corC}

\begin{corD}
 For $n \geq 5$, the group $H_n$ is residually $F_2$-free, that is, every element of the group survives in a quotient that does not contain a non-abelian free subgroup.
\end{corD}

Let us now turn to groups acting on CAT(0) square complexes that are not necessarily hyperbolic. Paralleling what we just did for acylindricity, we will be interested in the following weakening of the WPD condition, which again has the advantage of dealing only with stabilisers of pairs of points:  

\begin{definition*}   

Let $G$ be a group acting on a geodesic metric space $X$ and let $g$ be an infinite order element with quasi-isometrically embedded orbits. We say that $g$ satisfies the \textit{weak WPD condition} (or that the action is \textit{weakly acylindrical in the direction of} $g$) if there exist constants $m, N \geq 0$ such that for every point $x$ of $X$, there are at most $N$ group elements fixing both $x$ and $g^{m}x$.
\end{definition*}

Our second main result is the following:

\begin{thmE}\label{thm:thmE}
 Let $G$ be a group acting by isometries on a CAT(0) square complex $X$. Let $g$ be a strongly contracting element for the CAT(0) metric and suppose that $g$ satisfies the weak WPD condition. Then $g$ satisfies the WPD condition. In particular, $G$ is either virtually cyclic or acylindrically hyperbolic.  
\end{thmE}

As a weakly acylindrical action is weakly acylindrical in the direction of each of its strongly contracting elements, we obtain the following corollary: 

\begin{corF}
 Let $G$ be a group acting weakly acylindrically on a CAT(0) square complex $X$. If $g$ is a strongly contracting element of $G$ for the CAT(0) metric, then $g$ satisfies the WPD condition. In particular, $G$ is either virtually cyclic or acylindrically hyperbolic.  
\end{corF}

Notice that the Rank Rigidity Theorem for CAT(0) cube complexes of Caprace--Sageev \cite[Theorem A]{CapraceSageevRankRigidity} provides us with a way to show the existence of strongly contracting isometries. In particular, we obtain the following corollary: 

\begin{corG}
 Let $G$  be a group acting weakly acylindrically on a CAT(0) square complex $X$ such that: 
 \begin{itemize}
  \item the action is essential,
  \item the action does not have a fixed point in $X \cup \partial_\infty X$,
  \item the complex $X$ is not the product of two unbounded trees.
 \end{itemize}
  Then $G$ contains strongly contracting elements for the CAT(0) metric, and every such element satisfies the WPD condition. In particular, $G$ is either virtually cyclic or acylindrically hyperbolic.
\end{corG}

 For instance, this corollary can be used to recover the acylindrical hyperbolicity of the Higman group on $4$ generators directly from the action on its associated (non-hyperbolic) CAT(0) square complex \cite{MartinHigmanCubical}. Indeed, such an action is  weakly acylindrical \cite[Corollary 3.6]{MartinHigmanCubical}, and the Rank Rigidity Theorem of Caprace--Sageev can be applied to show the existence of strongly contracting group elements \cite[Remark 2.2]{MartinHigmanCubical}. 
 
 The Higman group is thus a second natural example, in addition to the mapping class group of a closed hyperbolic surface, of a group acting weakly acylindrically on a polyhedral complex naturally associated to it. As weak acylindricity is a much more tractable condition, and a possibly strictly weaker one, it is natural to ask the following:

 \begin{question}
  Which well-known group actions, which might not be (known to be) acylindrical, are weakly acylindrical?
 \end{question}

 Moreover, it is natural to ask whether such characterisations of acylindricity and of the WPD condition hold for more general classes of spaces. We thus conclude this introduction with the following: 
 
 \begin{question}\begin{itemize}
  \item For which classes of hyperbolic spaces is the weak acylindricity of an action equivalent to the acylindricity of that action? 
  \item For which classes of geodesic metric spaces does the weak WPD condition of a strongly contracting isometry implies the WPD condition for that isometry?
  \end{itemize}
 \end{question}

Let us explain briefly the organisation of the article. In a nutshell, the strategy we follow is the following: For a given $r>0$, to every pair of points $x, y$ in $X$ and every group element $g$ moving $x$ and $y$ by at most $r$, we associate a geodesic quadrangle between the points $x, y, gx, gy$, and we fill this loop by using an appropriate disc diagram. The aim is then to show that, if $x$ and $y$ are sufficiently far apart, then such `filling surfaces' must have large portions in common. In particular, the associated group elements will necessary act the same way on a very long common geodesic segment, which is where the weaker forms of acylindricity considered in this paper enter the picture. 

Thus, after recalling standard results about the geometry of CAT(0) square complexes and disc diagrams in Section \ref{sec:preliminaries}, we study in detail the combinatorial geometry of disc diagrams in Section \ref{sec:diagrams}. Section \ref{sec:grids} studies the way grids (that is, CAT(0) square complexes isometric to Euclidean rectangles tiled by unit squares) can be mapped to a given CAT(0) square complex. With these tools at hand, we prove the main theorems in Section \ref{sec:proofs}. Finally, Section \ref{sec:Higman} is devoted to the geometry of generalised Higman groups.
 
 \begin{thx}
 The author thanks his colleagues at the University of Vienna, and in particular Federico Berlai and Markus Steenbock, for discussions on Higman's group that motivated this work. 
 This work was partially supported by the European Research Council (ERC) grant no. 259527 of Goulnara Arzhantseva and by the Austrian Science Fund (FWF) grant M1810-N25.
 \end{thx}

\section{Preliminaries on CAT(0) square complexes}\label{sec:preliminaries}

This section contains standard results about CAT(0) square complexes which will be used in this article. Throughout this section, $X$ will be a CAT(0) square complex. 

\subsection{Disc diagrams and the combinatorial Gau\ss--Bonnet Theorem}

A \textit{disc diagram} $D$ over $X$ is a contractible planar square complex endowed with a combinatorial map  $D \ra X$ that is an embedding on each square. A disc diagram $D$ over $X$ is called \textit{reduced} if no two distinct squares of $D$ that share an edge are mapped to the same square of $X$. We start by an elementary observation.

\begin{lem}\label{lem:disc_CAT0}
 Let $\varphi:D \ra X$ be a reduced disc diagram. Then $D$ is a CAT(0) square complex. 
\end{lem}

\begin{proof}
 Since the disc diagram is reduced, the restriction to the link of any vertex is a local isometry. In particular, it sends a simple closed loop of that look to a simple closed loop in the link of the image. As $X$ is a CAT(0) square complex by assumption, simple closed loops in the links of vertices have at least $4$ edges. Thus simple closed loops in the links of vertices of $D$ have at least $4$ edges, hence $D$ is a CAT(0) square complex. 
\end{proof}

 A disc diagram is \textit{non-degenerate} if its boundary is homeomorphic to a circle, and is \textit{degenerate} otherwise.
Recall that, by the Lyndon--van Kampen Theorem, one can associate to every non-backtracking loop $S \ra X$ a reduced disc diagram whose restriction to the boundary is the given loop. 

\begin{definition}[disc diagram between geodesics, quadrangle]
 Let $\gamma_-$, $\gamma_+$ be two geodesics of $X$, with vertices $u_-, v_-$ and $u_+, v_+$ respectively. A \textit{disc diagram between} $\gamma_-$ and $\gamma_+$ consists of the following data: 
 \begin{itemize}
  \item geodesic paths in $\gamma_u$, $\gamma_v$ in $X$ between $u_-, u_+$ and $v_-, v_+$ respectively,
  \item a reduced disc diagram $\varphi:D \ra X$ whose boundary decomposes as the concatenation of $4$ paths $P_+, P_v, P_-, P_u$ such that $\varphi$ sends $P_+, P_v, P_-, P_u$ bijectively to $\gamma_+, \gamma_v, \gamma_-, \gamma_u$.
 \end{itemize}
The boundary paths $P_+,  P_-$ are called the \textit{upper side} and \textit{lower side} of $D$ respectively. The boundary paths $ P_v,  P_u$ are called the \textit{gates} of $D$. 
A non-degenerate disc diagram between two geodesics is called a \textit{quadrangle}.

More generally, given a planar CAT(0) square complex $D$ homeomorphic to a disc and a decomposition of its boundary into geodesic segments  $P_+, P_v, P_-, P_u$ as above, we say that $D$ is a square complex  \textit{between $P_-$ and $P_+$}.
\end{definition}

\begin{figure}[H]
\begin{center}
 \scalebox{0.95}{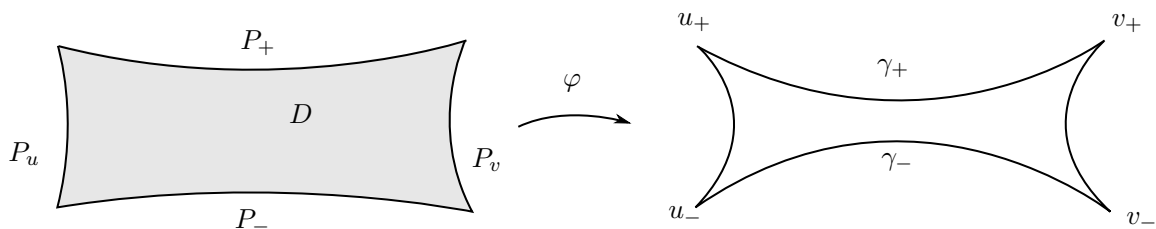}
\caption{A quadrangle.}
\label{fig_quadrangle}
\end{center}
\end{figure}

We now explain our main tool in controlling the geometry of reduced disc diagrams. Given a planar contractible square complex $D$, the \textit{curvature} of a vertex $v$ of $D$ is given by:
$$\kappa_D(v) = 2\pi - \pi \cdot \chi(\mbox{link}(v)) - n_v \frac{\pi}{2},$$
where $n_v$ denotes the number of squares of $D$ containing $v$.
A vertex of $D$ is  \textit{internal} if its link is connected and is called a \textit{boundary vertex} otherwise. The curvature of an internal vertex of $D$ is non-positive by Lemma \ref{lem:disc_CAT0}. The curvature of a boundary vertex is $\frac{\pi}{2}$ if it is contained in a single square of $D$, and non-positive otherwise. We call a boundary vertex of $D$ a \textit{corner} if it has non-zero curvature. The following version of the combinatorial Gau\ss--Bonnet Theorem follows the presentation of McCammond--Wise \cite[Theorem 4.6]{McCammondWiseFansLadders}.

\begin{thm}[Combinatorial Gau\ss-Bonnet Theorem] Let $D$ be planar contractible square complex. Then:
\begin{equation*}
~~~~~~~~~~~~~~~~~~~~~~~~~~~~~~~~~~~~\underset{v \mathrm{ ~vertex ~of~ }D}{\sum} \kappa_D(v) = 2\pi. ~~~~~~~~~~~~~~~~~~~~~~~~~~~~~~~~~~~~\qed
\end{equation*}
\label{GaussBonnet}
\end{thm}

\subsection{Hyperplanes in a CAT(0) square complex}

We briefly recall some notations and an elementary results about hyperplanes in a CAT(0) square complex.  A \textit{hyperplane} is a connected subspace of $X$ which intersects each square of $X$, isometrically identified  with $[-1,1] \times [-1,1]$, either in the empty set or in a segment of the form $\{0\} \times [-1,1]$ or $[-1,1] \times \{0\}$. An edge of $X$ intersecting a given hyperplane is said to be \textit{dual} to that hyperplane. Reciprocally, one can associate to every edge $e$ of $X$  a unique hyperplane meeting $e$, and such a hyperplane is said to be \textit{dual} to $e$. The \textit{combinatorial hyperplane} associated to a given hyperplane (also referred to in the literature as the \textit{carrier} of the hyperplane) is the minimal subcomplex of $X$ containing that hyperplane. Equivalently, it is the reunion of all the faces of $X$ containing an edge dual to that hyperplane. A hyperplane separates $X$ in exactly two connected components. A minimal subcomplex of $X$ containing one of these components is called a \textit{combinatorial half-space}.

In this article, we will only deal with combinatorial hyperplanes, and by a slight abuse of notations, we will denote by $H$ combinatorial hyperplanes and by $H_e$ the combinatorial hyperplane (associated to the hyperplane) dual to a given edge $e$ of $X$. We recall the following standard result (see for instance \cite[Lemma 13.4]{HaglundWiseSpecial}):

\begin{lem}\label{lem:hyperplanes_convex}
 Combinatorial hyperplanes and combinatorial half-spaces of a CAT(0) square complex are combinatorially convex.\qed
\end{lem}

We also have the following:

\begin{lem}\label{lem:halfspaces_CAT0convex}
Combinatorial half-spaces are convex for the CAT(0) metric.
\end{lem}

\begin{proof}
They are locally convex for the CAT(0) metric, and locally convex connected subsets of  a CAT(0) space are globally convex (see for instance  \cite[Theorem 1.10]{BuxWitzelLocallyConvex} for a proof of this fact). 
\end{proof}

The following useful result is well-known, see for instance \cite[Lemma 13.1]{HaglundWiseSpecial}:

\begin{lem}\label{lem:distinct_edges_hyperplane}
 A combinatorial geodesic of a CAT(0) square complex does not contain two distinct edges that define the same hyperplane.\qed
\end{lem}

If $\varphi: D \ra X$ is a reduced disc diagram, then $D$ is a CAT(0) square complex by Lemma \ref{lem:disc_CAT0}. The \textit{rails} of a combinatorial hyperplane of $D$ are the two maximal subcomplexes of that combinatorial hyperplane which we do not contain an edge dual to the associated hyperplane.

\subsection{Combinatorial intervals}

Recall that the \textit{combinatorial interval} between two vertices $v$, $v'$ of the CAT(0) square complex $X$, which we denote $\mbox{Int}_X(v,v')$, is the minimal subcomplex of $X$ containing all the combinatorial geodesics between $v$ and $v'$. Equivalently, it is the full subcomplex of $X$ associated to the reunion of all the combinatorial geodesics between $v$ and $v'$. The following is a direct consequence of Lemma \ref{lem:distinct_edges_hyperplane}:

\begin{lem}\label{lem:intervals_halfspaces}
 The combinatorial interval between two vertices of $X$ is the intersection of all the combinatorial half-spaces containing these two vertices.\qed
\end{lem}

Thus, Lemma \ref{lem:halfspaces_CAT0convex} implies the following: 

\begin{cor}\label{cor:intervals_CAT0convex}
Combinatorial intervals are convex for the CAT(0) metric. \qed
\end{cor}

 We also recall the following result (see \cite[Theorem 1.16]{PropertyACAT(0)CubeComplexes}):

\begin{lem}\label{lem:intervals_embed}
 A combinatorial interval of a  CAT(0) square complex isometrically embeds in $\bbR^2$ with its standard square tiling.\qed
\end{lem}

\section{Disc diagrams and isometric embeddings in the Euclidean plane}\label{sec:diagrams}

As disc diagrams will be our main tool in proving the results presented in the introduction, this section is devoted to study the combinatorial geometry of a reduced disc diagram $\varphi: D \ra X$. The goal of this section is threefold: to obtain a useful criterion ensuring that the CAT(0) square complex $D$ isometrically embeds in $\bbR^2$ with its standard square tiling (Proposition \ref{prop:embedding_Euclidean}), to obtain a criterion ensuring that the disc diagram $\varphi:D \ra X$ is an isometric embedding (Proposition \ref{lem:isometric_embedding}), and to prove that for a sufficiently well behaved disc diagram, the map restricts to an isometric embedding on a large subcomplex (Proposition \ref{cor:embedding_euclidean}). Throughout this section again, $X$ will be a given CAT(0) square complex.

\subsection{Singularities and (almost) Euclidean quadrangles}

\begin{definition}[singularities, almost Euclidean and Euclidean quadrangles]
 A \textit{singularity} of a quadrangle $\varphi: D \ra X$ is one of the following:
 \begin{itemize}
  \item an internal vertex of negative curvature,
  \item a corner of curvature at most $-\pi$,
  \item a pair of consecutive corners of curvature $-\frac{\pi}{2}$ in the interior of one of the sides $P_+$ and $ P_-$.
 \end{itemize}
 A  quadrangle is called \textit{almost Euclidean} if it contains no singularity. It is called \textit{Euclidean} if in addition the boundary of $D$ does not contain two consecutive corners of curvature $-\frac{\pi}{2}$.

More generally, a planar square complex homeomorphic to a disc is \textit{Euclidean} if it contains no internal vertex of negative curvature, no boundary vertex of curvature at most $-\pi$, and no pair of consecutive corners of curvature $-\frac{\pi}{2}$.
\end{definition}

The following lemma, which  shows that a given quadrangle cannot be `too far' from being almost Euclidean, will be used in Section \ref{sec:proofs}.

\begin{lem}\label{lem:sigularities}
 Let $\varphi: D \ra X$ be a quadrangle. Then $D$ contains at most $4$ singularities.
\end{lem}

\begin{proof}
 Let us denote by $P_i, i=1, \ldots, 4,$ the four geodesic sides of the boundary of $D$. Since each $P_i$ is a combinatorial geodesic, its interior does not contain two consecutive corners of positive curvature (see \cite[Lemma 3.8]{MartinHigmanCubical} and its proof). In particular, it follows that
 $$ \sum_{v \in \mathring{P}_i} \kappa_D(v) \leq \frac{\pi}{2},$$
 and if $P_i$ contains $n_i$ singularities in its interior, a similar argument yields 
  $$ \sum_{v \in \mathring{P}_i} \kappa_D(v) \leq \frac{\pi}{2} - n_i\frac{\pi}{2}.$$
 Since an internal vertex of $D$ has negative curvature if and only if it is a singularity, we have 
  $$ \sum_{v \in \mathring{D}} \kappa_D(v) \leq - n\frac{\pi}{2},$$
  where $n$ is the number of internal singularities of $D$. Finally, each of the remaining four vertices corresponding to the intersection of two adjacent sides $P_i, P_j$ brings a curvature of at most $\frac{\pi}{2}$. The combinatorial Gau\ss--Bonnet Theorem \ref{GaussBonnet} thus yields the following inequality: 
  $$2\pi \leq 4\cdot\frac{\pi}{2} + 4\cdot\frac{\pi}{2} - (\sum_{1\leq i \leq 4} n_i)\frac{\pi}{2} - n \frac{\pi}{2}$$
  and  the number $n + \sum_in_i$ of singularities of $D$ is thus bounded above by $4$.
\end{proof}

\subsection{Embeddability in the Euclidean plane}

\begin{prop}\label{prop:embedding_Euclidean}
 Let $D$ be a Euclidean square complex. Then $D$ embeds isometrically in $\bbR^2$ with its standard square tiling.
\end{prop}

\begin{proof}
 Let $H$ be a (combinatorial) hyperplane of $D$. Let $L_+, L_-$ be the two rails of $H$. We say that another hyperplane $H'$ \textit{osculates} $H$ if $H \cap H'$ is non-empty but does not contain any square of $D$.\\
 
 \textbf{Claim 1: } There exists at most one hyperplane of $D$ that osculates $H$ along $L_+$.\\
 
In order to prove this claim, first notice that if a square of $D$ meets $H$, it meets it along an edge, by the curvature condition on  Euclidean square complexes. For every hyperplane $H'$ of $D$ osculating $H$ along $L_+$, set
 $$J_{H'}:= H' \cap H.$$
By combinatorial convexity of hyperplanes in a CAT(0) square complex complex (Lemma \ref{lem:hyperplanes_convex}), it follows that each $J_{H'}$ is a sub-segment of $L_+$. Moreover, if two such osculating hyperplanes $H'$, $H''$ define sub-segments of $L_+$ that meet along a single vertex, then the curvature condition on almost Euclidean complex implies that $H' = H''$.  Suppose now by contradiction that there exist at least two hyperplanes of $D$ osculating $H$ along $L_+$. It follows from the previous discussion that we can choose a maximal non-empty sub-segment $J$ of $L_+$ such that no edge of $J$ is contained in a hyperplane of $D$ osculating $H$ along $L_+$, and such that the extremities $v_1$ and $v_2$ of $J$ are contained in two hyperplanes $H_1$ and $H_2$ respectively which osculate $H$ along $L_+$. By maximality of $J$, the curvature condition on almost Euclidean quadrangles implies that each $v_i$ is boundary vertex with curvature $-\frac{\pi}{2}$. By definition of $J$, there exists no corner of $D$ between $v_1$ and $v_2$, which contradicts the fact that a Euclidean square complex does not contain two consecutive corners of curvature $-\frac{\pi}{2}$. This concludes the proof of Claim 1. \\

Let $H_+$ and $H_-$ be the hyperplanes of $D$ osculating $H$ along $L_+$ and $L_-$ respectively (the argument is similar if there exists only one or zero such osculating hyperplane).  Choose isometric embeddings
$$\psi: H \hra \bbR^2, \psi_+: H_+ \hra \bbR^2 \mbox{ and } \psi_-: H_-\hra \bbR^2$$

such that the images of $\psi$, $\psi_-$ and $\psi_+$ are contained in distinct horizontal hyperplanes of $\bbR^2$ and such that 
$$\psi(H\cap H_-) = \psi_-(H\cap H_-), \psi(H\cap H_+) = \psi_+(H\cap H_+).$$
From the previous discussion on curvature in almost Euclidean quadrangles, it also follows that 
$$\psi(H)\cap \psi_-(H_-) = \psi_-(H\cap H_-), \psi(H)\cap \psi_+(H_+) = \psi_+(H\cap H_+).$$
Thus, we can glue these maps together into a combinatorial embedding 
$$ H_- \cup H \cup H_+ \hra \bbR^2$$
that is a local isometry, and which sends each of these combinatorial hyperplanes inside a horizontal hyperplane of $\bbR^2$ (with its usual square tiling). Reasoning by induction, we have that $D$ is covered by a finite sequence $H_i$ of hyperplanes, which we call \textit{horizontal}, such that, for distinct $i$ and $j$, we have that  $H_i \cap H_{j}$ is empty if $j \neq i \pm 1$ and is contained in the (unique) rail common to $H_i$ and $H_j$ otherwise. By induction, we thus obtain a combinatorial embedding $\psi:D \hra \bbZ^2$ where each $H_i$ is mapped isometrically into some horizontal hyperplane of $\bbZ^2$. \\

\textbf{Claim 2:} The map $\psi:D \hra \bbZ^2$ is an isometric embedding.\\ 

To prove this claim, let $P$ be a combinatorial geodesic in $D$, which we can write as a finite concatenation 
$$P= P_0 e_1 P_1 e_2 P_2 \ldots$$
where each $e_i$ is an edge of $D$ defining one of the horizontal hyperplanes, and $P_i$ is a segment contained in one of the rails of some horizontal hyperplane of $D$. Thus, $\psi(P)$ consists of the concatenation 
$$\psi(P)= \psi(P_0)\psi(e_1)\psi(P_1)\psi(e_2)\psi(P_2) \ldots$$
of horizontal segments $\psi(P_i)$ of $\bbZ^2$ and vertical edges $\psi(e_i)$. Choosing an orientation for $P$ yields an orientation for each $e_i$ and $P_i$ and hence for their images under $\psi$. Showing that $\psi(P)$ is an oriented geodesic of $\bbZ^2$ thus amounts to showing that all the $\psi(e_i)$ have the same orientation (up or down) and that all the $\psi(P_i)$ have the same orientation (left or right). By contradiction let us assume that this is not the case. Up to an isometry of $\bbZ^2$ preserving the set of its horizontal hyperplanes, we thus have two cases to consider: \\

\begin{figure}[H]
\begin{center}
 \scalebox{1}{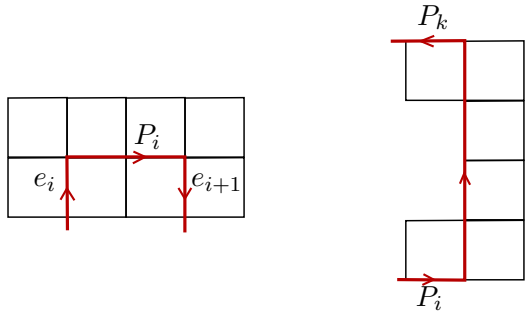}
\caption{Two impossible configurations.}
\label{fig_embedding_Z2}
\end{center}
\end{figure}

\textbf{Case 1:} $P$ contains a sub-segment $e_i P_i e_{i+1}$ such that $\psi(e_i)$ and $\psi(e_{i+1})$ have different orientations (one up, one down, as depicted on the left in Figure \ref{fig_embedding_Z2}). But by construction of $\psi$, this implies that both $e_i$ and $e_{i+1}$ define the same horizontal hyperplane of $D$, which contradicts Lemma \ref{lem:distinct_edges_hyperplane}. \\

\textbf{Case 2:} $P$ contains a sub-segment of the form $P_i e_{i+1} e_{i+2} \ldots e_k P_k$, for some $i<k$, such that $\psi(P_i)$ and $\psi(P_k)$ have different orientations (one left, one right, as depicted on the right in Figure \ref{fig_embedding_Z2}). For $i+1 \leq j \leq k$, if the square of $\bbZ^2$ to the left of $\psi(e_j)$ is not contained in $\psi(D)$, then $e_j$ is contained in the boundary of $D$. If such a square is contained in $\psi(D)$ for every $i+1 \leq j \leq k$, then the last edge of $P_i$ and the first edge of $P_k$ define the same (hyperplane), contradicting Lemma \ref{lem:distinct_edges_hyperplane}. Otherwise, choose a maximal sub-segment of $e_{i+1}\ldots e_k$ such that for each edge of it, the square of $\bbZ^2$ on its left is not in $\psi(D)$. Then its endpoints are boundary vertices of curvature $-\frac{\pi}{2}$ as $D$ is a non-degenerate disc diagram, and no other corner is between these two corners, contradicting the fact that a Euclidean square complex does not contain two consecutive corners of curvature $-\frac{\pi}{2}$. \\

This concludes the proof of Claim 2, and thus finishes the proof of Proposition \ref{prop:embedding_Euclidean}.
\end{proof}

\subsection{Isometric embeddings}

\begin{prop}\label{lem:isometric_embedding}
A  Euclidean disc diagram $\varphi: D \ra X$ is an isometric embedding.
\end{prop}

We will need the following result \cite[Lemma 2.11]{HaglundWiseSpecial}: 

\begin{lem}\label{lem:embedded_diagram}
 Let $\varphi:X_1 \ra X_2$ be a combinatorial immersion between two CAT(0) square complexes, and assume that the link of a vertex $v$ of $X_1$ is sent (injectively) to a full subgraph of the link of $\varphi(v)$. Then  $\varphi$ is an isometric embedding.  \qed
\end{lem}

\begin{proof}[Proof of Lemma \ref{lem:isometric_embedding}]
 By Proposition \ref{prop:embedding_Euclidean}, we can assume that $D$ is a subcomplex of $\bbR^2$.  We now describe an algorithm to \textit{complete} the Euclidean disc diagram $\varphi: D \ra X$ into a Euclidean disc diagram $\varphi': D' \ra X$, where $D'$ is a subcomplex of $\bbR^2$ containing $D$ such that the inclusion $D \hra D'$ is an isometric embedding, and such that for every vertex $w$ of $D'$, the induced map on the link of $w$ sends the link of $w$ in $D'$ to a full subgraph of $X$. 
 
 We proceed by induction and construct maps $\varphi_i:D_i \ra X$ where $(D_i)$ is an increasing sequence of non-degenerate sub-diagram of $\bbR^2$, such no $D_i$  contains two consecutive corners of negative curvature, and such that each inclusion $D_i \hra D_{i+1}$ is an isometric embedding. Set $D_0:= D$ and $\varphi_0:= \varphi : D_0 \ra X$. Suppose that $D_0, \ldots, D_n$ and $\varphi_0, \ldots, \varphi_n$ have been constructed. If $D_n$ contains a vertex $w$ such that the map $\varphi_n$ sends the link of $w$ to a non-full subgraph of $X$, then, since the disc diagram is non-degenerate, necessarily $w$ is a corner of curvature $-\frac{\pi}{2}$. Let $C$ be the unique square of $\bbR^2$ containing $w$ which is not contained in $D_n$. Then $C \cap D_n$ consists of exactly two edges, since $D_n$ does not contain two consecutive corners of negative curvature by the inductive hypothesis. Thus, there is a unique way to extend $\varphi_n:D_n \ra X$ to $\varphi_{n+1}:D_{n+1} \ra X$, where $D_{n+1}:= D_n \cup C$. Moreover, the inclusion $D_n \hra D_{n+1}$ is an isometric embedding,  $D_{n+1}$ does not contain two consecutive corners of negative curvature, and  $\varphi_{n+1}:D_{n+1} \ra X$ is reduced.   
 
 This algorithm eventually stabilises at some stage $N \geq 0$. Indeed, if we choose a big square of $\bbR^2$ containing $D$, then it contains also every $D_i$.  The map $D \ra X$ thus factorises as $D \hra D_N \overset{\varphi_N}{\ra} X$. Moreover, as $\varphi_N$ is reduced by construction, it is automatically an immersion, as $D_N$ is a subcomplex of $\bbR^2$. Thus, $\varphi_N$ is an isometric embedding by Lemma \ref{lem:embedded_diagram}, and the same holds for $D \ra X$. 
\end{proof}

\subsection{Euclidean quadrangles in disc diagrams}

\begin{definition}[width]
 A quadrangle $\varphi:D \ra X$ between geodesic segments $\gamma_-, \gamma+$ of $X$ is \textit{of width at most $r$} if its  sides are at Hausdorff distance at most $r$ in $D$. 
 
 Analogously, a planar CAT(0) square complex $D$, homeomorphic to a disc, between segments $P_-$ and $P_+$ is \textit{of width at most $r$} if $P_-$ and $P_+$ are at Hausdorff distance at most $r$ in $D$.
\end{definition}

\begin{prop}\label{cor:embedding_euclidean}
 For every $r \geq 0$ , there exists a constant $L_{\mathrm{emb}}(r)\geq 0$  such that the following holds: Let $\varphi:D \ra X$ be an almost Euclidean quadrangle, of width at most $r$, between two geodesics $\gamma_-$ and $\gamma_+$, and let $\gamma_-'\subset \gamma_-$ be a sub-segment at distance at least $L_{\mathrm{emb}}(r)$ from the endpoints of $\gamma_-$.  Then there exists a Euclidean sub-quadrangle $D'$ of $D$, of width at most $r$, between  $\gamma_-'$ and a sub-segment $\gamma_+' \subset \gamma_+$.
\end{prop}

\begin{proof}
Since vertices of almost Euclidean quadrangles have  uniformly bounded valence by definition, there exists a constant $N(r)$ such that every combinatorial ball of radius $r$ in a non-degenerate almost Euclidean quadrangle contains at most $N(r)$ edges. Moreover, there exists a constant $L_1(r)$ such that the following holds:

Let $\varphi:D \ra X$ be an almost Euclidean disc diagram of width at most $r$. Let $v_1, v_2$ be the vertices of $\gamma_-$ at distance $2N(r+1)+1$ from the endpoints $u_1, u_2$ of $\gamma_-$. Since $D$ has width at most $r$, the combinatorial ball of radius $r$ around $v_1$ or $v_2$ disconnects $D$. In particular, a rail of a hyperplane crossing an edge of $\gamma_-$ between $u_i$ and $v_i$ which does not cross an edge of $B(v_i, r)$ contains at most $L_1(r)$ edges.

Now set 
$$L_{\mathrm{emb}}(r):=(2N(r+1)+1)+ L_1(r).  $$

Let $\varphi:D \ra X$ be a non-degenerate almost Euclidean (reduced) disc diagram of width at most $r$, between two geodesics $\gamma_-$ and $\gamma_+$, with $|\gamma_-| \geq 2L_{\mathrm{emb}}(r)$. Let $v_1, v_2$ be the vertices of $\gamma_-$ at distance $2N(r+1)+1$ from the endpoints $u_1, u_2$ of $\gamma_-$. By Lemma \ref{lem:distinct_edges_hyperplane}, for each $i=1, 2$, there exists an edge $a_i$ of $\gamma_-$ between $u_i$ and $v_i$ such that the combinatorial hyperplane $H_{a_i}$ does not meet  $B(u_i, r) \cup B(v_i, r)$. In particular, each $H_{a_i}$ crosses $\gamma_+$. Moreover, $H_{a_1}$ and $H_{a_2}$ are disjoint as in addition each $B(v_i,r)$ disconnects $D$. Thus, $H_{a_1}$ and $H_{a_2}$ defines a Euclidean sub-quadrangle $D'$ of $D$ whose gates are contained in the rails of some combinatorial hyperplanes  (see Figure \ref{fig_Euclidean}). \\

\begin{figure}[H]
\begin{center}
 \scalebox{0.75}{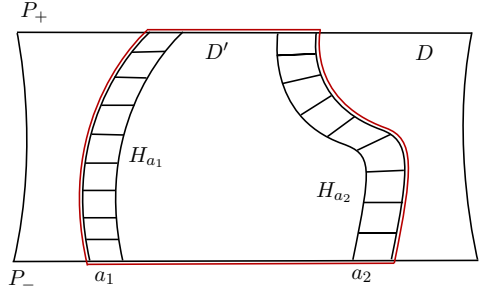}
\caption{The construction of the Euclidean sub-quadrangle $D'$.}
\label{fig_Euclidean}
\end{center}
\end{figure}

In particular, $\varphi_|:D' \ra X$ is a Euclidean quadrangle, and the restriction of $\varphi$ to $D'$ isometrically embeds in $\bbR^2$ by \ref{prop:embedding_Euclidean}. Since the rails of each $H_{a_i}$ have length at most $L_1(r)$, and $\gamma_-$ is at distance at least $L_1(r)$ from $a_1$ and $a_2$,  we can choose a sub-quadrangle $\varphi'':D'' \ra X$ of width at most $r$ between  $\gamma_-' $ and a sub-segment $\gamma_+'' \subset \gamma_+$, which concludes the proof. 

\end{proof}

\section{Grids and their combinatorial geometry}
\label{sec:grids}

In this section, we study the way certain subcomplexes of the Euclidean plane with its standard tiling can be mapped to a given CAT(0) square complex. Throughout this section, $X$ is a given CAT(0) square complex. 

\begin{definition}[grid]\label{def:grid}
A \textit{grid} of $D$ is a CAT(0) square complex isometric to $I_l \times I_h$, $l \geq 1$, $h \geq 0$, where $I_l$ and $I_h$ denote a simplicial segment on $l$ and $h$ edges respectively. Such a grid is said to be of \textit{length} $l$ and \textit{width} $h$.

A \textit{grid of $X$} is a reduced disc diagram $\varphi: D \ra X$.
\end{definition}

\begin{rmk}
 It follows from Lemma \ref{lem:embedded_diagram} that a grid of $X$ is an isometric embedding. Therefore, we will sometimes use the term `grid of $X$' to denote the image of such a disc diagram.
\end{rmk}

 \begin{prop}\label{lem:factorisation_grid}
  Let $\varphi:I_m \times I_n \ra X$ be a combinatorial map such that  $\varphi_{I_m \times \{0\}}$ is a combinatorial geodesic. Then $\varphi$ factorises as 
  $$ I_m \times I_n \xrightarrow{id \times \varphi_n} I_m \times T_n \xhookrightarrow{h} X,$$
  where $T_n$ is a simplicial tree, $\varphi_n: I_n \ra T_n$ is a combinatorial map, and $  I_m \times T_n \hra X$ is an isometric embedding.
 \end{prop}
 
 \begin{proof}
  We prove the result by induction on $n$, for $m$ fixed. The result for $n=1$ is immediate as combinatorial hyperplanes embed isometrically in a CAT(0) cube complex \cite[Theorem 4.10]{SageevCubeComplex}. 
  Suppose the result is true at rank $n \geq 1$ and consider a map $\varphi:I_m \times I_{n+1} \ra X$ satisfying the assumptions of the lemma. The restriction of $\varphi$ to $I_m \times I_n$ factorises by the induction hypothesis. Denote by $C_{i,j}$ the square of $I_m \times I_{n+1}$ on the $i$-th column (starting from the left) and the $j$-th line (starting from the bottom).  
  
  If there exists a $1 \leq i_0 \leq m$ such that  $\varphi(C_{i_0,n+1})= \varphi(C_{i_0,j_0})$, then $\varphi(C_{i,n+1})= \varphi(C_{i,j_0})$ for every $1 \leq i \leq m$. Indeed, this follows by induction, by starting from $C_{i_0,j_0}$ and moving through adjacent squares in the $j_0$-th line, since, in the CAT(0) square complex $X$, two squares sharing two adjacent edges are the same.
  Thus, if there exists a $1 \leq i_0 \leq m$ such that  $\varphi(C_{i_0,n+1})= \varphi(C_{i_0,j_0})$, then $\varphi$ factorises as 
  $$I_m \times I_{n+1}  \xrightarrow{id \times \varphi_{n+1}} I_m \times T_n \hra X,$$
  where  $I_m \times T_n \hra X$ comes from the induction hypothesis, $\varphi_{n+1}: I_{n+1} \ra T_n$ is the unique combinatorial map such that $id \times \varphi_{n+1}$ restricts to $id \times \varphi_{n}$ on $I_m \times I_n$ and is such that $(id \times \varphi_{n+1})(C_{i,n+1})= (id \times \varphi_{n})(C_{i,j_0})$ for every $1 \leq i \leq m$.
  
  If there does not exist  a $1 \leq i_0 \leq m$ such that  $\varphi(C_{i_0,n+1})= \varphi(C_{i_0,j_0})$, then write $I_{n+1}$ as the reunion of $I_n$ and an edge $e$ glued along a vertex $v$, and consider the tree $T_n\cup_{\varphi_n(v)} e'$ obtained by attaching to $T_n$ an edge $e'$ along $\varphi_n(v)$. Then $\varphi$ factorises as 
  $$I_m \times I_{n+1}  \xrightarrow{id \times \varphi_{n+1}} I_m \times (T_n\cup_{\varphi_n(v)} e' )  \xrightarrow{h'} X,$$
  where $\varphi_{n+1}:I_{n+1} \ra T_n\cup_{\varphi_n(v)} e'$ is the unique combinatorial map that restricts to $\varphi_n$ on $I_n$ and sends $e$ to $e'$, and $h'$ restricts to $h$ on $I_m \times T_n$. By construction, the map $I_m \times (T_n\cup_{\varphi_n(v)} e' ) \xrightarrow{h'} X$ is a local isometry, i.e. it is injective on the link of vertices, as $h$ already is. Moreover, as links of vertices of $X$ have girth at least $4$, it follows that for every vertex $v$ of  $I_m \times (T_n\cup_{\varphi_n(v)} e' )$, the image of the link of $v$ under $h'$ is a full subgraph of the link of $h'(v)$. Thus, $h'$ is an isometric embedding by Lemma \ref{lem:embedded_diagram}.
 \end{proof}

\section{Proof of the main theorems}
\label{sec:proofs}

This section is devoted to the proof of Theorems A and E. The proofs being extremely similar, we start with the slightly more technical Theorem E and explain in Section \ref{sec:ThmA} how to adapt the proof to deduce Theorem A. 
 
\subsection{Strongly contracting axes}
\label{sec:axes}
A CAT(0) square complex has two metrics naturally associated it, depending on whether we are considering the $\ell_1$- or $\ell_2$-metric on its squares. In the $\ell_1$-case, we recover the combinatorial distance on the set of vertices, and in the $\ell_2$-case we recover the CAT(0) metric. In what follows, we will indicate with subscripts which metric is being considered, whether talking about the translation length $|\cdot|_i$, the distance $d_i(\cdot, \cdot)$, etc.

For the rest of Sections \ref{sec:axes}, \ref{sec:corridor} and \ref{sec:finish}, we assume that we are given a CAT(0) square complex $X$, a group $G$ acting on it by isometries, and an element $g \in G$ which is strongly contracting (for the $\ell_2$-metric) and satisfies the weak WPD condition. An \textit{$\ell_2$-axis} of $g$ will mean a geodesic line of $X$ (for the $\ell_2$-metric) on which $g$ acts by translation. An \textit{$\ell_1$-axis} (or \textit{combinatorial axis}) of $g$ will mean a geodesic line of the $1$-skeleton of $X$ (for the $\ell_1$-metric) on which $g$ acts by translation. Combinatorial axes of hyperbolic isometries of CAT(0) square complexes were shown to exist by Haglund \cite{HaglundSemiSimple}. We choose, for $i= 1$ or $2$, an $\ell_i$-axis $\Lambda_i$  for $g$.

\begin{deflem}[tubular constant]\label{def:tubular}
Since $\Lambda_2$ is strongly contracting by assumption, \cite[Corollary 3.4]{BestvinaFujiwaraHigherRank} implies that we can choose a \textit{tubular constant} $C_2$ such that for every two points $x, y$ of $X$ which project to two points of $\Lambda_2$ at distance at least $C_2$, then any geodesic from $x$ to $y$ meets the $C_2$-neighbourhood of $\Lambda$.\qed
\end{deflem}

\begin{deflem}[the constant $\delta$]\label{lem:distance_axes}
 There exists a constant $\delta>0$ such that $\Lambda_1$ and $\Lambda_2$ are at Hausdorff distance at most $\delta$ from one another (for both the $\ell_1$- and the $\ell_2$-metric). 
\end{deflem}

\begin{proof}
By \cite[Lemma 3.8]{BestvinaFujiwaraHigherRank}, there exists a constant $C_2'$ such that for every pair of points $u, v$ in the $C_2$-neighbourhood of $\Lambda_2$, a ball disjoint from the CAT(0) geodesic $[u,v]$ projects on  $[u,v]$ to a subset of diameter strictly less than $C_2'$.   Choose integers $m, m' \geq 1$ such that $m|g|_2 > C_2$ and $m'|g|_2 > C_2' + |g|_1+2$.  Let $x_1$ be a point of $\Lambda_1$ and let $x_2$ be its projection  on $\Lambda_2$ (and thus $g^{2m+m'}x_2$ is the projection of $g^{n+2m}x_1$ on $\Lambda_2$).

Let $u$ be a point of the CAT(0) geodesic between $x_1$ and $g^{m'+2m}x_1$ which projects to the point $g^mx_2 \in \Lambda_2$. Since $d_2(x_2,g^mx_2)=m|g|_2>C_2$, there exists a point $y$ in the sub-segment between $x$ and $u$ and a point $y_2\in \Lambda_2$ in between $x_2$ and $g^mx_2$ such that $d_2(y,y_2) \leq C_2$. Analogously, let $v$ be a point of the geodesic between $x$ and $g^{m+2m'}x_1$ which projects to the point $g^{m+m'}x_2 \in \Lambda_2$. Since $d_2(g^{m'+m}x_2,g^{m'+2m}x_2)=m|g|_2>C_2$, there exists a point $z$ in the sub-segment between $v$ and $g^{m'+2m}x$ and a point $z_2\in \Lambda_2$ in between $g^{m'+m}x_2$ and $g^{m'+2m}x_2$ such that $d_2(z,z_2) \leq C_2$. By convexity of the CAT(0) metric, it follows that the geodesic between $y$ and $z$ is contained in the $C_2$-neighbourhood of $\Lambda_2$.

Let us now consider the combinatorial interval $\mbox{Int}_{X}(x_1, g^{m'+2m}x_1)$ between $x_1$ and $g^{m'+2m}x_1$. By Corollary \ref{cor:intervals_CAT0convex} and Lemma \ref{lem:intervals_embed}, it is a convex subcomplex (for the CAT(0) metric) and we can think of it as a subcomplex of $\bbR^2$ with its standard tiling. Let $K$ be the subset of $\bbR^2$ consisting of those points whose orthogonal projection on the bi-infinite line generated by the segment $[y, z]$ is contained in $[y, z]$. It follows from the definition of $C_2'$ that  $\mbox{Int}_{X}(x_1, g^{m'+2m}x_1) \cap K$ is contained in the $C_2'$-neighbourhood of $[y, z]$. Thus, every combinatorial geodesic from $x_1$ to $g^{m'+2m}x_1$ contains a sub-interval of $\ell_2$-length at least $d_2(y,z) \geq |g|_1+2$, and thus a sub-segment of $\ell_1$-length at least $|g_1|$, which is $C_2'$-close to $[y,z]$. In particular, the axis $\Lambda_1$ contains a sub-segment of $\ell_1$-length at least $|g_1|$ which is $(C_2' + C_2)$-close to $\Lambda_2$ for the $\ell_2$-metric, and thus $\Lambda_1$ and $\Lambda_2$ are at Hausdorff distance at most $C_2' + C_2$.
\end{proof}

An immediate corollary is the following:

\begin{defcor}[the constant $C_\Box$]\label{def:C_Box}
 There exists a constant $C_\Box$ such that  every grid whose interior is disjoint from the axis $\Lambda_1$ projects on $\Lambda_1$ (for the $\ell_1$-metric) with a diameter strictly smaller than $C_\Box$.\qed
\end{defcor}

\begin{lem}\label{lem:distance_translates}
 Let $\gamma_1$ be a geodesic segment contained in $\Lambda_1$ and let $h$ be a group element that moves the endpoints of $\gamma_1$ by a distance of at most $r$ (for the $\ell_1$-metric). Let $\gamma_1'$ be a sub-segment of $\gamma_1$. Then $\gamma_1'$ and $h\gamma_1'$ are at Hausdorff distance at most $2r + 8\delta$ (for the $\ell_1$-metric).
\end{lem}

\begin{proof}
 By lemma \ref{lem:distance_axes}, $\Lambda_1$ and $\Lambda_2$ are at Hausdorff distance at most $\delta$ for the $\ell_1$-metric and for the $\ell_2$-metric. Let $\gamma_2, \gamma_2'$ be the $\ell_2$-projections,  of $\gamma_1$, $\gamma_1'$ on the convex subset $\Lambda_2$. The subsets $\gamma_1$ and $\gamma_2$ ($\gamma_1'$ and $\gamma_2'$ respectively) are at Hausdorff distance at most $\delta$ for the $\ell_2$-metric by construction. We thus have:
  \begin{align*} 
 d_1(\gamma_1',h\gamma_1') &\leq  d_1(\gamma_1', \gamma_2') +  d_1(\gamma_2',h\gamma_2') +  d_1(h\gamma_2', h\gamma_1')  ,\\
  &\leq 4d_2(\gamma_1', \gamma_2') +  d_1(\gamma_2',h\gamma_2'),\\
  &\leq 4\delta + 2 d_2(\gamma_2',h\gamma_2')\\
   &\leq 4\delta + 2(r + 2\delta) \leq 2r + 8\delta,
 \end{align*}
 the last inequalities following from the convexity of the CAT(0) metric, since $h$  moves the endpoints of $\gamma_1$ by  at most $r$.
\end{proof}

\subsection{Staircases}

 \begin{definition}
  A Euclidean CAT(0) square complex $D=D(P_-, P_+)$ between geodesic segments $P_-$ and $P_+$ is called a \textit{staircase} if $P_-$ contains at least one corner of $D$ in its interior.
 \end{definition}
 
 \begin{lem}\label{lem:staircase}
  Let $D=D(P_-, P_+)$ be a  Euclidean CAT(0) square complex of width at most $r$ between geodesic segments $P_-$ and $P_+$ of $D$, and $\varphi_i:D \ra X$ a family of combinatorial maps that all coincide on $P_-$. Let $k \geq 1$. If $P_-$ contains at least $2r + 2k$ corners, then the $\varphi_i$ coincide on a sub-geodesic of $P_+$ of length at least $k$.
 \end{lem}

\begin{proof}
 Recall that $D$ embeds in $\bbR^2$ with its standard square tiling by Lemma \ref{lem:intervals_embed}. Thus, we will assume that $D$ is a subcomplex of $\bbR^2$.
 
 Up to an isometry of $\bbR^2$, we can assume that the following holds: For every horizontal (respectively vertical) edge of $P_-$, the unique square above it (respectively to its left) is in $D$ (see Figure \ref{fig_staircase}).
 
 Let $v$ and $v'$ be the first and last corners of $P_-$ of curvature $-\frac{\pi}{2}$. We will prove that the $\varphi_i$ coincide on 
 $$K:= D \cap \mbox{Int}_{\bbR^2}(v,v'),$$
 where $\mbox{Int}_{\bbR^2}(v,v')$ is the combinatorial interval between $v$ and $v'$.
 We order the squares of $K$ into a sequence $C_1, C_2 \ldots$ by starting from the downmost horizontal hyperplane of $K$ and reading from right to left; Once a horizontal hyperplane has been exhausted, move to the one above it and apply the same procedure.\\
 
\begin{figure}[H]
\begin{center}
 \scalebox{1}{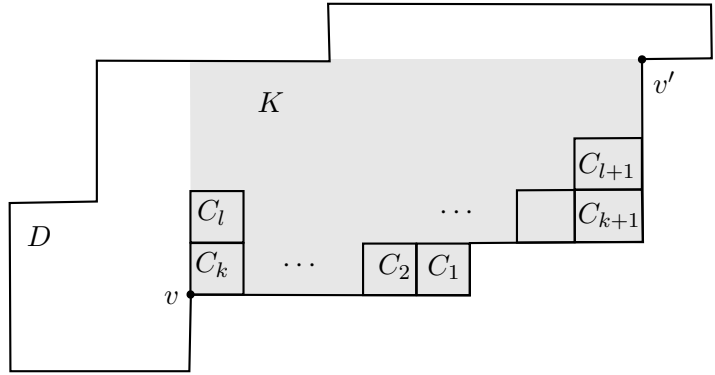}
\caption{Ordering the squares of $K$. The subcomplex $K$ is the shaded region.}
\label{fig_staircase}
\end{center}
\end{figure}
 
 We now show by induction that the $\varphi_i$ coincide on $K_n:= \cup_{1\leq i \leq n}C_i$. This holds for $n=1$. Indeed, two edges of $C_1$ are in $P_-$. As the $\varphi_i$ coincide on $P_-$, the CAT(0) condition implies that they also coincide on $C_1$. Suppose that the $\varphi_i$ coincide on $K_n$ for $n\geq 1$. Then the right and bottom edges of $C_{n+1}$ belong to $P_- \cup K_n$, and the same reasoning thus implies that the $\varphi_i$ coincide on $C_{n+1}$.

 Since $P_-$ contains at least $2r + 2k$ corners, we can choose a vertex $w$ in $P_-$ such that there exist at least $r+k$ corners between $w$ and $v$, and between $w$ and and $v'$. It thus follows that 
 $$K \supset D \cap B_{\bbR^2}(w,r + k),$$
 where $B_{\bbR^2}(\cdot)$ denotes the ball in $\bbR^2$ for the combinatorial distance. As the geodesic $P_+$ meets the $r$-neighbourhood of $w$ by assumption, it follows that $K$ contains a sub-geodesic of $P_+$ of length at least $k$, and the $\varphi_i$ thus coincide on a sub-geodesic of $P_+$ of length at least $k$.
\end{proof}

\subsection{Corridors over the axis $\Lambda_1$}
\label{sec:corridor}

\begin{definition}[Corridor]\label{def:almost_flat}
Let $D=D(\gamma_-, \gamma_+)$ be 
a Euclidean quadrangle. We say that $D$ is a \textit{corridor} if every vertex in the interior of $P_-$ and $P_+$ has zero curvature in $D$.
\end{definition}

\begin{rmk}
 The WPD condition for $g$ is equivalent to requiring that there exist constants $L$ and $N$ such that there do not exist $N$ distinct group elements that fix pointwise two points of $\Lambda_1$ at distance at least $L$. We thus choose two such constants $L, N \geq 0$ for the remaining of this section.
\end{rmk}

The aim of this section is to prove the following intermediate result: 

\begin{prop}\label{lem:corridor}
 For every $r>0$, there exist constants $L_0(r), N_0(r)$ such that the following holds: For every $x, y$ in $\Lambda_1$, for every sub-segment $\gamma$ of $\Lambda_1$ contained in the sub-segment of $\Lambda_1$ between $x$ and $y$, for every sub-segment $\gamma_- \subset \gamma$ of length at least $L_0(r)$, there exist at most $N_0(r)$ elements $h$ of $G$ such that $d(x,hx), d(y,hy)<r$, and such that  there exist a quadrangle between $\gamma$ and $h \gamma$ 
together with a sub-quadrangle between $\gamma_-$ and a sub-segment $\gamma_+ \subset h\gamma$ which is a corridor of width at most $4r+4C_\Box + 16\delta$.
\end{prop}

The reason for the constant $4r+4C_\Box + 16\delta$ will become apparent in the next section. A key tool in controlling  such quadrangles is the following operations of `concatenating' and `piling up' disc diagrams.

\begin{definition}[concatenation of disc diagrams]\label{def:concatenation_diagram}
 Let $\varphi:D \ra X$ ( $\varphi':D' \ra X$ respectively) be a quadrangle between two combinatorial geodesic $\gamma_-$ and $\gamma_+$ ($\gamma_-'$ and $\gamma_+'$ respectively). Assume that $\gamma_+$ and $\gamma_-'$ intersect along a non-empty geodesic sub-segment. 
 
The two disc diagrams $\varphi:D \ra X$ and $\varphi':D' \ra X$ can be \textit{concatenated} into a disc diagram 
 $\Phi: D \cdot D' \ra X$
where $D\cdot D'$ is the planar contractible square complex obtained from the disjoint union of $D$ and $D'$ by identifying a point $x$ in the upper side of $D$ with a point $x'$ in the lower side of $D'$ precisely when $\varphi(x) = \varphi'(x')$. 
\end{definition}

\begin{definition}[piling up of disc diagrams]
Let $\gamma$ be a combinatorial geodesic of $X$.
Let $(h_i)$ a sequence of group elements and denote by $\varphi_{h_i}: D(h_i) \ra X$ a quadrangle between  $\gamma$ and $h_i\gamma$.
For each integer $i \in \bbZ$, the map 
$ h_1\dots h_i \varphi_{h_{i+1}} $ defines a disc diagram between  $h_1\ldots h_i\gamma$ and $h_1\ldots h_{i+1}\gamma$. Thus, for each $n$ we can successively concatenate the disc diagrams $ h_1\dots h_i \varphi_{h_{i+1}} $ into a disc diagram between $\gamma$ and $h_1 \cdots h_n \gamma$. 
We say that this disc diagram is obtained by \textit{piling up} the sequence of disc diagrams $(\varphi_{h_i}:D(h_i)\ra X)_{1 \leq i \leq n}$. 
\end{definition}

We now start the proof of Proposition \ref{lem:corridor}, which in splits in several steps. Throughout this proof, whenever $x, y$ are vertices of $\Lambda_1$, $\gamma \subset \Lambda_1$ is a sub-segment, and $\gamma_- \subset \gamma$ is a sub-segment, we will say that an element $g$ of $G$ satisfies the \textit{property} $(P_r)$ \textit{with respect to} $x,y, \gamma, \gamma_-$  if $d(x,gx), d(y,gy)<r$, and there exists a quadrangle $D(g)$ between $\gamma$ and $g\gamma$ and a sub-quadrangle $D'(g)$ between $\gamma_- $ and a sub-segment $\gamma_+ \subset g\gamma$ which is a corridor. The proof is split in several steps.\\

\textbf{Step 1.} Let $x, y$ be vertices of $\Lambda_1$ at distance at least 
$2(4r + 4C_\Box+ 16\delta)$, $\gamma$ a sub-segment of $\Lambda_1$ between them, and $\gamma_- \subset \gamma$ a subsegment of length at least 
$2(4r + 4C_\Box+ 16\delta)$. Let $h$ be an element of $G$ satisfying the property $(P_r)$ with respect to $x,y, \gamma, \gamma_-$, and denote by $\varphi_h:D(h) \ra X$ an associated quadrangle and $D'(h)$ the associated sub-quadrangle which is a corridor. Let $k,n \geq 1$ be constants that will be chosen later. 

Let $\Phi_n:D^n(h) \ra X$ be the disc diagram between $\gamma$ and $h^n\gamma$ obtained by piling up $n$ copies of the disc diagram  $\varphi_h:D(h)\ra X$. 

By Lemma \ref{prop:embedding_Euclidean}, $D'(g)$ contains a sub-quadrangle which is a grid $G_1(g)$, of length $|\gamma_-|-2(4r + 4C_\Box+ 16\delta)$ and width  $l \leq 4r + 4C_\Box+ 16\delta$. We call this integer $l$ the \textit{width} of $g$. Moreover, by Lemma \ref{lem:distance_translates}, the upper side of $G_1(g)$ and the $h$-translate of the bottom side of $G_1(g)$ are at Hausdorff distance at most $(4r + 4C_\Box+ 16\delta) + (2r + 8\delta) = 6r + 4C_\Box + 24 \delta.$ If $|\gamma_-| > 2(4r + 4C_\Box+ 16\delta) + 2(i-1)(6r + 4C_\Box + 24 \delta)$ for some $i \geq 1$, then $D'(g)$ contains a sub-quadrangle which is a grid $G_i(g)$, of length $|\gamma_-|-(2(4r + 4C_\Box+ 16\delta) + 2(i-1)(6r + 4C_\Box + 24 \delta))$ and width $l \leq 4r + 4C_\Box+ 16\delta$, over the maximal sub-segment of $\gamma_-$ at distance at least $4r + 4C_\Box+ 16\delta + (i-1)(6r + 4C_\Box + 24 \delta)$ of the extremities of $\gamma_-$. 

As a consequence, if $|\gamma_-|> k+ 2(4r + 4C_\Box+ 16\delta) + 2(n-1)(6r + 4C_\Box + 24 \delta)$, then $D^n(g)$ contains a sub-quadrangle $G^n(g)$ between a sub-segment of $\gamma_-$ and a sub-segment of $g^n\gamma_-$, which is a grid $k \times nl$. 
Such a grid $G^n(g)$ is obtained as the concatenation of sub-grids $G_{n,i}(g)$, where each $G_{n,i}(g)$ is a sub-grid $k \times l$ of $G_{n-1-i}(g)$ between a sub-segment of $\gamma$ and a sub-segment of $g\gamma$. We also choose a vertical path $P_n(g)$ of length $nl$ in  $G^n(g)$, which we write as the concatenation of $n$ vertical paths $P_{n,i}(g)$ of length $l$, each $P_{n,i}(g)$ belonging to $ G_{n,i}(g)$ (see Figure \ref{fig_pilingup}).

\begin{figure}[H]
\begin{center}
 \scalebox{0.8}{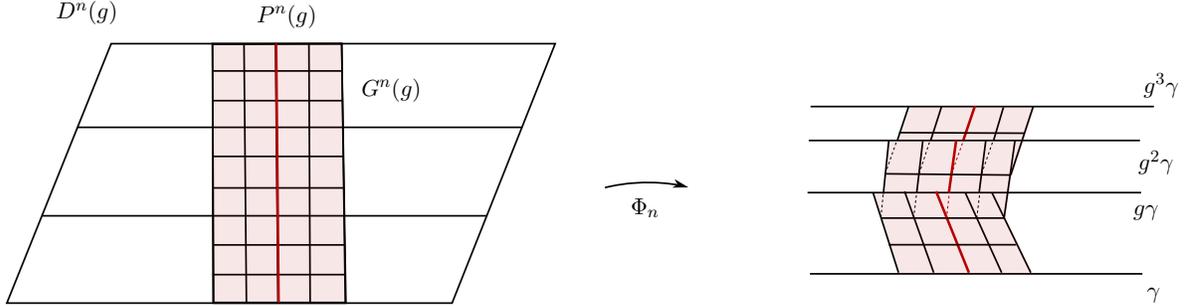}
\caption{Piling up $n=3$ copies of a sub-corridor of $D(g)$ of width $3$. Here, the (impossible) case where the paths $\Phi_n(P_{n,i}(g))$ and $\Phi_n(P_{n,i+1}(g))$ backtrack on exactly one edge.}
\label{fig_pilingup}
\end{center}
\end{figure}

By Proposition \ref{lem:factorisation_grid}, the restriction of $\Phi_n$ to $G^n(g)\cong I_k \times I_{nl}$ factorises as 
$$ I_k \times I_{nl} \overset{id \times \Psi_n}{\longrightarrow} I_k \times T_{nl} \overset{}{\hra X},$$

Note that for every $0 \leq i <n$, the paths $\Phi_n(P_{n,i}(g))$ and $\Phi_n(P_{n,i+1}(g))$ backtrack on the same number $l'$ of edges because of Proposition \ref{lem:factorisation_grid}. If we had $l' < \frac{l}{2}$, then the image  $\Phi_n(P_n)$ would contain an isometrically embedded path, and the length of such paths grows linearly with $n$. In particular for some constant $n_0 \geq 1$ that is independent of $x$ and $y$, this length would be bigger than the constant $C_{\Box}$ (Definition \ref{def:C_Box}). If $k$ is bigger than $C_{\Box}$, then it would follow that $X$ would contain an embedded square of side $C_{\Box}$ whose projection on $\gamma$ has diameter $C_{\Box}$, contradicting Lemma \ref{def:C_Box}. \\

\textbf{Step 2.} We thus start by  defining the following constant: 
$$L_0(r):= \big(\mbox{max}(L,C_{\Box}, |g|_1) + 36r+ 24C_\Box + 144\delta\big) + 2(4r+4C_\Box+ 16\delta) +2(n_0-1)(6r + 4C_\Box + 24 \delta + |g|_1),  $$
Let $x, y$ be vertices of $\Lambda_1$ at distance at least $L_0(r)$, $\gamma$ be the sub-segment of $\Lambda_1$  between them, and $\gamma_- \subset \gamma$ a subsegment of length at least $L_0(r)$. Let $h$ be a group element that satisfies the property $(P_r)$ with respect to $x,y, \gamma,$ and $ \gamma_-$. From the previous paragraph, and using the same notations, it follows that  $\Phi_n(P_{n,i}(h))$ and $\Phi_n(P_{n,i+1}(h))$ backtrack on at least $\frac{l}{2}$ edges. 
For $\alpha$ a real number between $0$ and the width of $h$, we denote by $\ell_{\alpha}(h)$ the horizontal line of $ G_1(h) $ at distance $\alpha$ from $P_-$, and by $\ell_\alpha'(h)$ the sub-segment $\ell_\alpha(h) \cap G_2(h)$.
Since $\Phi_n(P_{n,i}(h))$ and $\Phi_n(P_{n,i+1}(h))$ backtrack on at least $\frac{l}{2}$ edges for each $i$, it follows that 
$$h\varphi\big(\ell_{\frac{l}{2}}'(h)\big) \subset \varphi\big(\ell_{ \frac{l}{2} }(h)\big).$$
Thus, both $h\varphi\big(\ell_{\frac{l}{2}  }'(h)\big)$ and $\varphi\big(\ell_{ \frac{l}{2}  }'(h)\big)$ are contained in the segment $\varphi\big(\ell_{ \frac{l}{2}}(h)\big).$ 

We claim that $h\varphi\big(\ell_{\frac{l}{2}  }'(h)\big)$ and $\varphi\big(\ell_{ \frac{l}{2}  }'(h)\big)$ are at Hausdorff distance at most $6r + 4C_\Box + 24 \delta$. Indeed,  
this follows from the triangular inequality, as $\varphi(\ell_{ \frac{l}{2}  }'(h))$ is $(2r+6C_\Box)$-close to the segment $\varphi(\ell_0'(h))$ by the width assumption, the same holds for $h\ell_{ \frac{l}{2}  }'(h)$ and $h\ell_0'(h)$, and $\ell_0'(g)$ and $h\ell_0'(g)$ are at most $(2r+8\delta)$-close by Lemma \ref{lem:distance_translates}.

Thus,  $h\varphi\big(\ell_{\frac{l}{2}  }'(h)\big)$ and $\varphi\big(\ell_{ \frac{l}{2}  }'(h)\big)$ are contained in the segment $\varphi\big(\ell_{ \frac{l}{2}}(h)\big)$ and are at Hausdorff distance at most $6r + 4C_\Box + 24 \delta$ from one another. In particular, the two sub-segments share an edge, as they have a length of at least $L + 36r+ 24C_\Box + 144\delta$.\\

\textbf{Step 3.} Let us show that, for an element $h$ as in the previous step, the subset $\varphi\big(\ell_{ \frac{l}{2}  }'(h)\big)$ is contained in a combinatorial axis of $g$.

Note that $hg$ has translation length at most $|g|_1 + r$. 
By concatenating the disc diagrams $(hg)^i\varphi_h: D(h)\ra X$ for $i \geq 0$, the same reasoning as before shows that $hg\varphi_{h}(\ell_{l/2}'(h))$ and $\varphi_{h}(\ell_{l/2}'(h))$ are contained in a common segment and are at Hausdorff distance at most $|g|_1 + 6r + 4C_\Box + 24 \delta$. As we already now that $h\varphi_{h}(\ell_{l/2}'(h))$ and $\varphi_{h}(\ell_{l/2}'(h))$ are contained in a segment and are at Hausdorff distance at most $ 6r + 4C_\Box + 24 \delta$, it follows that $hg\varphi_{h}(\ell_{l/2}'(h))$ and $h\varphi_{h}(\ell_{l/2}'(h))$, and thus $g\varphi_{h}(\ell_{l/2}'(h))$ and $\varphi_{h}(\ell_{l/2}'(h))$, are contained in a segment and are at Hausdorff distance at most $|g|_1 +12r+8C_\Box+48\delta$. As the length of $\ell_{l/2}'(h)$ is strictly greater than $|g|_1 +12r+8C_\Box+48\delta$ by construction, it follows that these two segments share at least an edge. Moreover, they cannot coincide as $g$ is a hyperbolic isometry of $X$. Thus, the same argument as before shows that $\bigcup_{i\in \bbZ} g^i\varphi_{h}(\ell_{l/2}(h))$ is a combinatorial axis for $g$, and it contains $\varphi\big(\ell_{ \frac{l}{2}  }'(h)\big)$ by construction. \\

\textbf{Step 4.} Let $x, y$ be vertices of $\Lambda_1$, let $\gamma$ be the sub-segment of $\Lambda_1$ between them, and $\gamma_- \subset \gamma$ a sub-segment of length at least $L_0(r)$. Let us now consider two elements $h$ and $h'$ satisfying $(P_r)$ with respect to $x,y, \gamma$ and $ \gamma_-$ and having the same width $l \leq 4r+4C_\Box+16\delta$.  Denote by $\varphi_h: D(h) \ra X$ and $\varphi_{h'}: D(h') \ra X$ the associated disc diagrams. Consider elements $h_i$ of $G$, $i \geq 1$, such that $h_i=h$ if $i$ is odd and $h_i=h'$ otherwise. As before, we can construct the disc diagram $\Phi_n: D^n(h,h')\ra X$ obtained by piling up the disc diagrams $(\varphi_{h_i}:D(h_i)\ra X)_{1 \leq i \leq n}$.
As before, $D^n(h,h')$ contains a sub-diagram $G^n(h,h')$ between  a sub-segment of $\gamma_-$ and a sub-segment of $h_1\cdots  h_n\gamma_-$,    which is a grid $k \times nl$. Moreover, this map factorises by Proposition \ref{lem:factorisation_grid}, and we have to consider the image under $\Phi_n$ of a vertical path $P_n(h,h')$ of length $nl$, which is the concatenation of paths $P_{n,i}(h_i)\subset h_1\cdots h_iD(h_i)$.
Here again, note that for every $ i$, the paths $\Phi_n(P_{n,2i-2}(h_{2i-1}))$ and $\Phi_n(P_{n,2i-1}(h_{2i}))$ backtrack on the same number $l_{\mathrm{even}}$ of edges, while $\Phi_n(P_{2i-1}(h_{2i}))$ and $\Phi_n(P_{2i}(h_{2i+1}))$ backtrack on the same number $l_{\mathrm{odd}}$ of edges. Using the same reasoning, at least one of the two integers $l_{\mathrm{even}}$ and $l_{\mathrm{odd}}$ is at least $\frac{l}{2}$. Let us assume for instance that this holds for $l_{\mathrm{even}}$. As before, it then follows that 
$ \varphi_h(\ell_{ \frac{l}{2}  }'(h)) $ and $ (h')^{-1}\varphi_{h'}(\ell_{ \frac{l}{2} }'(h'))$ are at Hausdorff distance at most $6r + 4C_\Box + 24 \delta$. 
But as $ (h')^{-1}\varphi_{h'}(\ell_{ \frac{l}{2} }'(h'))$ and $ \varphi_{h'}(\ell_{ \frac{l}{2} }'(h'))$ are also at Hausdorff distance at most $6r + 4C_\Box + 24 \delta$ by Step 2, it follows that $ \varphi_h(\ell_{ \frac{l}{2}  }'(h)) $ and $ \varphi_{h'}(\ell_{ \frac{l}{2} }'(h'))$ are at Hausdorff distance at most $12r + 8 C_\Box + 48\delta$. \\

\textbf{Step 5.} We define the following constant:
$$N_0(r) := (4r+4C_\Box+16\delta+1)(12r + 8 C_\Box + 48\delta+1)N.$$
Let $x, y$ be vertices of $\Lambda_1$, $\gamma$ the sub-segment of $\Lambda_1$ between them, and $\gamma_- \subset \gamma$ a subsegment of length at least $L_0(r)$. Suppose by contradiction that there at least $N_0(r)$ distinct elements satisfying the property $(P_r)$ with respect to $x,y, \gamma$ and $ \gamma_-$. For each such element $h$ we consider an associated corridor $\varphi_h: D(h) \ra X$ and use the same notations as above. From Step 4, it follows that for every two such elements $h, h'$ with the same width $l$, the segments 
$ \varphi_h(\ell_{ \frac{l}{2}  }'(h))$ and $ \varphi_{h'}(\ell_{ \frac{l}{2} }'(h'))$ are at Hausdorff distance at most $12r + 8 C_\Box + 48\delta$ from one another, and are contained in a segment of $X$.
Choose  $(12r + 8 C_\Box + 48\delta+1)N$ such elements  of the same width. It follows that the intersection $\ell$ of all the segments $ \varphi_h(\ell_{ \frac{l}{2}  }'(h))$ has length at least $L + 12r + 8 C_\Box + 48\delta= (L + 36r+ 24C_\Box + 144\delta) - 2(12r + 8 C_\Box + 48\delta)$. Let $\ell'$ be the maximal sub-segment of $\ell$ at distance $6r+4C_\Box + 24\delta$ from the endpoints of $\ell$. Each of the $(12r + 24 C_\Box + 16\delta+1)N$ elements send $\ell'$ to a translate contained in $\ell$ by Step 2. Thus, we can choose at least $N$ distinct such elements whose action on $\ell'$  coincide. As $\ell'$ is of length at least $L$ and is contained in a combinatorial axis of $g$ by Step 3, we obtain a contradiction. This finishes the proof of Proposition \ref{lem:corridor}.

\subsection{Finishing the proof of Theorem E}
\label{sec:finish}

\begin{lem}\label{lem:width}
Let $\gamma$ be a segment of $\Lambda_1$  of length at least $ 2C_{\Box}$, and let $h$ be a group element that moves the endpoints of $\gamma$ by at most $r$. Suppose that there exists a reduced disc diagram $\varphi:D \ra X$ between $\gamma$  and $h\gamma$. Then $D$ has width at most $2r+4C_\Box $.
\end{lem}

\begin{proof}
Recall that $D$ is a CAT(0) square complex by Lemma \ref{lem:disc_CAT0}. Let $I$ be the combinatorial interval between the endpoints of the lower path of $D$, and consider the restriction  $\varphi_|: I \ra X$. If $\varphi_|:I \ra X$ is non-degenerate, then it is a Euclidean disc diagram, as $I$ isometrically embeds in $\bbR^2$ by Lemma \ref{lem:intervals_embed}. In such a case, $\varphi_|$ is an isometric embedding by Lemma \ref{lem:isometric_embedding}. If $\varphi_|:I \ra X$ is degenerate, then the restriction of $\varphi_|$ to each of the maximal subcomplexes of $I$ homeomorphic to a $2$-disc is an isometric embedding by the same argument. It follows from Proposition \ref{prop:embedding_Euclidean} and Lemma \ref{def:C_Box} that two combinatorial geodesics of $X$ between the endpoints of $\gamma$ are at Hausdorff distance at most  $2C_\Box$ from another. Thus, any two combinatorial geodesics of $D$ between the endpoints of the upper path of $D$ are at Hausdorff distance at most $2C_\Box$. 

Recall that combinatorial intervals are convex for the CAT(0) metric by Corollary \ref{cor:intervals_CAT0convex}. Thus,  the CAT(0) geodesic $P_+'$ between the endpoints of the upper  path $P_+$ of $D$ is at Hausdorff distance at most $2C_{\Box}$ from $P_+$. Analogously, the CAT(0) geodesic $P_-'$ between the endpoints of the lower path $P_-$ of $D$ is at Hausdorff distance at most $2C_{\Box}$ from $P_-$. But by convexity of the CAT(0) metric, $P_-'$ and $P_+'$ are at Hausdorff distance at most $r$ for the $\ell_2$-metric. Thus, $P_-$ and $P_+$ are at Hausdorff distance at most $2r + 4C_\Box$ for the $\ell_1$-metric.
\end{proof}

Recall that the interval between two vertices of a CAT(0) square complex embeds isometrically in $\bbR^2$ endowed with its square tiling (Lemma \ref{lem:intervals_embed}). We can thus define the following constant:

\begin{definition}
  For every $R \geq 0$, we choose a constant $N_{\mathrm{int}}(R)$ such that there exist at most $N_{\mathrm{int}}(R)$ combinatorial geodesics between two vertices of $X$ at distance at most $R$. 
\end{definition}

Further recall that a Euclidean quadrangle $D = D(\gamma_-, \gamma_+)$ isometrically embeds in $\bbZ^2$ by Lemma \ref{lem:intervals_embed}. We can thus define the following constant:

 \begin{definition}
 For every $R \geq 0$, we choose a constant $N_{\mathrm{quad}}(R,r)$ on the number of  Euclidean quadrangles $D = D(\gamma_-, \gamma_+)$ of width at most $4r+4C_\Box + 16\delta$ and such that $|\gamma_-| \leq R$,  up to isometries which preserve pointwise the upper lower sides of the quadrangles. 
 \end{definition}

We are now ready to prove Theorem E. We split the proof in several steps. The first $7$ steps deal with points on the axis $\Lambda_1$. Note that, for such points, the acylindricity condition amounts to proving that there exist constants  $L_{\mathrm{axis}}(r), N_{\mathrm{axis}}(r)$ such that for every two points $x,y$ of $\Lambda_1$ at least $L_{\mathrm{axis}}(r)$ apart, at most $N_{\mathrm{axis}}(r)$ group elements move $x$ and $y$ by at most $r$. 

Let $r >0$. We start by defining the following constants:
  \begin{align*} 
 L_1(r) &:= 2L_{\mathrm{emb}}(4r+4C_\Box + 16\delta) + (24r + 8 C_\Box + 78 \delta + 2L)(L_0(r)+2) ,\\
  L_2(r) &:= 5\cdot L_1(r),\\
  N_1(r) &:= (8r+32\delta+1)N_{\mathrm{quad}}(L_1(r))N,\\
  N_2(r) &:= \mbox{max}(N_0(r),N_1(r)),\\
  N_{3}(r)  &:= (8r+32\delta+1)L(r)^2N_{\mathrm{int}}(L(r))N,\\
  N_{4}(r)  &:= 5 \cdot 2^{24r + 8 C_\Box + 78 \delta + 2L}  N_2(r).
 \end{align*}
 Finally, we define the two constants: 
   \begin{align*} 
  L_{\mathrm{axis}}(r)&:= 3\cdot L_2(r),\\
  N_{\mathrm{axis}}(r) &:= 4  \mbox{max}(N_{\mathrm{3}}(r),N_{\mathrm{4}}(r)).
 \end{align*}

We now show that, for every vertices $x,y$ of $\Lambda_1$ at distance at least $L_{\mathrm{axis}}(r)$, there do not exist $N_{\mathrm{axis}}(r)$ distinct elements $h$ of $G$ such that $d(x,hx),d(y,hy) \leq r$. 

By contradiction, suppose that there exist vertices $x,y$ of $\Lambda_1$ at distance at least $L_{\mathrm{axis}}(r)$ and $N_{\mathrm{axis}}(r)$  distinct group elements $h$ of $G$ such that $d(x,hx),d(y,hy) \leq r$. For simplicity,  we arrange these elements as a finite sequence $(h_i)$. Choose a taut geodesic $\gamma$ between $x$ and $y$. \\

\textbf{Step 1.} Decompose $\gamma$ as the concatenation of geodesic segments
$$ \gamma := \alpha \cup \gamma_1 \cup \gamma_2 \cup \gamma_3 \cup \beta,$$
such that 
each of the segments $\gamma_i$ has length  $L_2(r)$. We sort the chosen group elements into four classes, depending on the nature of both $h_i\gamma \cap \gamma_1$ and $h_i\gamma \cap \gamma_3$: empty or non-empty. Since $ N(r) = 4  \mbox{max}(N_{3}(r),N_{4}(r))$, we can choose at least $\mbox{max}(N_{3}(r),N_{4}(r))$ such elements in the same class. 

Assume by contradiction that all the $h_i\gamma \cap \gamma_1$ and $g_i\gamma \cap \gamma_3$ are non-empty. Then since we are considering at least $L(r)^2 \cdot (8r+32\delta+1)N_{\mathrm{int}}(L(r))N$ group elements, we can choose at least $ (8r+32\delta+1)N_{\mathrm{int}}(L(r))N$ such elements such that the $h_i\gamma$ all contain two fixed vertices $v_1 \in \gamma_1$ and $v_3 \in \gamma_3$. Now the sub-segment of $h_i\gamma$ between $v_1$ and $v_3$ is in $\mbox{Int}_X(v_1,v_3)$, which contains at most $N_{\mathrm{int}}(L(r))$ by construction, so we can find at least $(8r+32\delta+1)N$ elements such that the $h_i\gamma$ all contain a chosen combinatorial geodesic $\gamma'$ between $v_1$ and $v_3$. It thus follows that the segments $h_i^{-1}\gamma'$ are contained in $\gamma$ and are at Hausdorff distance at most $4r+16\delta$ of one another by Lemma \ref{lem:distance_translates}. Since there are at least $(8r+32\delta+1)N$ such elements $h_i^{-1}$, it thus follows that the action of at least $N$ of them coincide on $\gamma'$, and thus $N$ distinct elements of $G$ fix pointwise $\gamma'$, which is of length at least $L_2(r) \geq L$, a contradiction.\\

\textbf{Step 2.} From the previous step, we can thus assume that, for each of the  given elements $h_i$, the segments $h_i \gamma$ and $\gamma_1$, and in particular  $h_i \gamma_1$ and $\gamma_1$, are disjoint.  Moreover, $h_i \gamma_1$ and $\gamma_1$ are at  Hausdorff distance at most $2r+8 \delta$ by Lemma \ref{lem:distance_translates}. 
Thus, we choose for each $i$ a  quadrangle $D(h_i)$, of width at most $4r+4C_\Box + 16\delta$ by Lemma \ref{lem:width}, between $\gamma_1$ and $h_i \gamma_1$.

Decompose $\gamma_1$ as the concatenation of geodesic segments
$$ \gamma_1 :=  \gamma_{1,1} \cup \ldots \cup \gamma_{1,5} ,$$
such that  each of the $5$ segments $\gamma_i$ has length $L_1(r)$.
We subdivide each $h_i \gamma_1$ into a concatenation 
$$ h_i \gamma_1 :=  \gamma_{1,1}^i \cup \ldots \cup \gamma_{1,5}^i,$$
such that one can subdivide the non-degenerate quadrangle $D(h_i)$ into a concatenation of $5$ non-degenerate quadrangles:
$$ D(h_i) := D_1(h_i) \cdots  D_{5}(h_i),$$
where each $D_k(h_i)$ is a non-degenerate sub-quadrangle between $\gamma_{1,k}$ and $\gamma_{1,k}^i$ which has width at most $4r+4C_\Box + 16\delta$. \\

\textbf{Step 3. } By Lemma \ref{lem:sigularities}, for each $i$ at least one of these non-degenerate sub-quadrangles is almost Euclidean. Now, since we have at least $N_{\mathrm{non-deg}}(r)= 5 \cdot 2^{24r + 8 C_\Box + 78 \delta + 2L}N_2(r)$ such elements, there exists an integer  $1 \leq k\leq 5$ such that, for at least $2^{24r + 8 C_\Box + 78 \delta + 2L}N_2(r)$ of these group elements, the sub-quadrangle $D_{k}(h_i)$ is almost Euclidean. \\

\textbf{Step 4.} Since $\gamma_{1,k}$ is of length  
$$ L_1(r) = 2L_{\mathrm{emb}}(4r+4C_\Box + 16\delta) + (24r + 8 C_\Box + 78 \delta + 2L)(L_0(r)+2),$$
we can choose a subsegment $\widetilde{\gamma}_{1,k}$ of $\gamma_{1,k}$, of length $(24r + 8 C_\Box + 78 \delta + 2L)(L_0(r)+2)$ and at distance $L_{\mathrm{emb}}(4r+4C_\Box + 16\delta)$ from the endpoints of $\gamma_{1,k}$. For each $i$, Proposition \ref{cor:embedding_euclidean} allows us to choose a sub-segment $\widetilde{\gamma}_{1,k}^i$ of $\gamma_{1,k}^i$ and a \textit{Euclidean} sub-quadrangle $\widetilde{D}_{k}(h_i)$ of $D_{k}(h_i)$ of width at most $4r+4C_\Box + 16\delta$ between  $\widetilde{\gamma}_{1,k}$ and $\widetilde{\gamma}_{1,k}^i$.\\

\textbf{Step 5.} Now write $\widetilde{\gamma}_{1,k}$ as the concatenation of  $24r + 8 C_\Box + 78 \delta + 2L$ sub-segments $\widetilde{\gamma}_{1,k,l}$ of length  $L_0(r)+2$. For each $i$ and each $l$, we can choose a sub-segment 
$\widetilde{\gamma}_{1,k,l}^i$ of $\widetilde{\gamma}_{1,k,l}^i$ and a sub-quadrangle $\widetilde{D}_{k,l}(h_i)$ of $\widetilde{D}_{k}(h_i)$ of width at most $4r+4C_\Box + 16\delta$ between $\widetilde{\gamma}_{1,k,l}$ and $\widetilde{\gamma}_{1,k,l}^i$.   Among the given $2^{24r + 8 C_\Box + 78 \delta + 2L}N_2(r)N$ group elements, we can now then choose $N_2(r)N$ of them such that for each $l$, the associated sub-quadrangles $(\widetilde{D}_{1,k,l}(h_i))_i$ are all of the same \textit{shape}: corridor or staircase. \\

\textbf{Step 6.} If for some $l$, all the $(\widetilde{D}_{1,k,l}(h_i))_i$  were corridors, then the fact that $N_2(r)\geq N_0(r)$ would yield a contradiction with Lemma \ref{lem:corridor}, so let us assume that they all are staircases. In particular, for each such $i$, $\widetilde{D}_{1,k}(h_i)$ contains at least $24r + 8 C_\Box + 78 \delta + 2L$ corners on $\widetilde{\gamma}_{k,1}$. Since we have at least $(4r+2)N_
{\mathrm{quad}}(r)N$ group elements, at least $(4r+2)N$ of them define isometric sub-quadrangles. Since $\widetilde{D}_{1,k}(h_i)$ has width at most $4r+12C_\Box$ and contains at least 
$$24r + 8 C_\Box + 78 \delta + 2L= 2(4r+4C_\Box + 16\delta) + 2(L + 2(4r + 16 \delta))$$
corners, it follows from Lemma \ref{lem:staircase} that there exists a geodesic segment $P$ of length $L + 2(4r + 16 \delta)$ contained in all the $g_i\gamma$. \\

\textbf{Step 7.} It now follows that the segments $g_i^{-1}P$ are contained in $\gamma$ and are at Hausdorff distance at most $4r+16\delta$ of one another by Lemma \ref{lem:distance_translates}. Since there are at least $(8r+32\delta+1)N$ such elements $g_i^{-1}$, it follows that the action of at least $N$ of them coincide on a geodesic segment of length $L$, hence $N$ distinct elements of $G$ fix pointwise a sub-segment of length $L$ of $\Lambda_1$, a contradiction. \\

This concludes the proof that for vertices $x,y$ of the axis $\Lambda_1$ at distance at least $L_{\mathrm{axis}}(r)$, at most $N_{\mathrm{axis}}(r)$ elements of $G$ move $x$ and $y$ by a distance of at most $r$. In the final step, we deal with arbitrary points of $X$.\\

\textbf{Step 8.} 
Recall that the tubular constant $C_2$ was introduced in Definition \ref{def:tubular}. Let $m$ be an integer such that $m|g|_2> C_2$. Let $n$ be an integer that will be fixed later, let $x$ be a vertex of $X$, and let $h$ be a group element such that $d_1(x,hx), d_1(g^{n+2m}x,hg^{n+2m}x) \leq r$. Consider the (unique) CAT(0) geodesic between $x$ and $g^{n+2m}x$. Let also $x_2$ be the projection of $x$ on $\Lambda_2$ (and thus $g^{n+2m}x_2$ is the projection of $g^{n+2m}x$ on $\Lambda_2$).

Let $u$ be a point of the geodesic between $x$ and $g^{n+2m}x$ which projects to the point $g^mx_2 \in \Lambda_2$. Since $d_2(x,g^mx)=m|g|_2>C_2$, there exists a point $y$ in the sub-segment between $x$ and $u$ and a point $y_2\in \Lambda_2$ in between $x_2$ and $g^mx_2$ such that $d_2(y,y_2) \leq C_2$. Analogously, let $v$ be a point of the geodesic between $x$ and $g^{n+2m}x$ which projects to the point $g^{m+n}x_2 \in \Lambda_2$. Since $d_2(g^{n+m}x_2,g^{n+2m}x_2)=m|g|_2>C_2$, there exists a point $z$ in the sub-segment between $v$ and $g^{n+2m}x$ and a point $z_2\in \Lambda_2$ in between $g^{n+m}x_2$ and $g^{n+2m}x_2$ such that $d_2(z,z_2) \leq C_2$. Since $\Lambda_1$ and $\Lambda_2$ are at Hausdorff distance at most $\delta$ by Definition-Lemma \ref{lem:distance_axes}, choose points $y_1$, $z_1$ on $\Lambda_1$ which are $\delta$-close from $y_2$ and $z_2$ respectively. 

Since  $d_1(y,hy), d_1(g^{n+2m}z,hg^{n+2m}z) \leq r$ by convexity of the $\ell_2$-metric, it follows that $$d_1(y_1,hy_1), d_1(z_1,hz_1) \leq 2r+ 2\delta + 4 C_2.$$ Moreover, we have 
  \begin{align*} 
  d_2(y_1,z_1)&\geq d_2(x_2,g^{n+2m}x_2) - d_2(x_2,y_2) - d_2(y_2,y_1) - d_2(z_1,z_2) - d_2(z_2,g^{n+2m}x_2)\\
   &\geq  n|g|_2 - 2\delta,
 \end{align*}
 and thus $d_1(y_1,z_1) \geq n|g|_2 - 2\delta$.  Thus, we first choose an integer $n(r)$ (which is independent of the point $x$) such that  $ n(r)|g|_2 - 2\delta \geq L_{\mathrm{axis}}(2r+2\delta+4C_2)$. Now set:
   \begin{align*} 
  m(r)&:= n(r) + 2m,\\
  N(r) &:= N_{\mathrm{axis}}(2r+2\delta+4C_2).
 \end{align*}
 It thus follows that for every $x$ of $X$, there exist at most $N(r)$ group elements that move $x$ and $g^{m(r)}x$ by at most $r$, for otherwise we could find two points of $\Lambda_1$ at distance at least $L_{\mathrm{axis}}(2r+2\delta+4C_2)$ which are moved by at most $2r+2\delta+4C_2$ by $N_{\mathrm{axis}}(2r+2\delta+4C_2)$ distinct group elements, a contradiction.

\subsection{Deducing Theorem A}
\label{sec:ThmA}

In this section, we explain how to modify the proof of Theorem E to obtain Theorem A. The proof is almost identical: the proof of Theorem E took place `along the axis' of a strongly contracting WPD element, which displays features of hyperbolic geometry. Under the hypotheses of Theorem A, the whole complex $X$ is hyperbolic, and the following lemmas can be adapted to deal with arbitrary geodesics of $X$, rather than sub-segments of the combinatorial axis considered in the previous section:\\

Lemma \ref{def:tubular} still holds in general hyperbolic metric spaces: For \textit{any} geodesic $\gamma$ of length at least $8\delta'$, $\delta'$ the hyperbolicity constant of the space, and any two points $x, y$ of the space that project on $\gamma$ at distance at least $8\delta'$, then any geodesic from $x$ to $y$ meets the $8\delta'$-neighbourhood of $\gamma$ (this follows for instance from the Tree Approximation Threorem for hyperbolic spaces \cite[Theorem 8.1]{CoornaertDelzantPapadopoulos}).

Lemma \ref{lem:distance_axes} is reformulated as follows: For every combinatorial geodesic $\gamma$ of the hyperbolic CAT(0) square complex $X$, the unique CAT(0) geodesic between the endpoints of $\gamma$ stays $2C_\Box$-close to $\gamma$, where the no-square constant $C_\Box$ is reinterpreted as a constant such that $X$ does not contain isometrically embedded $C_\Box\times C_\Box$ grids.

Analogously, Lemma \ref{lem:distance_translates} is reformulated as follows:  Let $\gamma$ be a geodesic segment and let $h$ be a group element that moves the endpoints of $\gamma$ by a distance of at most $r$ (for the $\ell_1$-metric). Let $\gamma'$ be a sub-segment of $\gamma$. Then $\gamma'$ and $h\gamma'$ are at Hausdorff distance at most $2r + 16C_\Box$ (for the $\ell_1$-metric).

Finally, Lemma \ref{lem:width} is reformulated as follows: Let $\gamma$ be a combinatorial geodesic of $X$ of length at least $ 2C_{\Box}$, and let $h$ be a group element that moves the endpoints of $\gamma$ by at most $r$. Suppose that there exists a quadrangle $\varphi: Q \ra X$ between $\gamma$  and $h\gamma$. Then $Q$ has width at most $2r+4C_\Box $.\\

With these modifications, we can prove Theorem A by adapting the proof of Theorem E presented in Sections \ref{sec:corridor} and \ref{sec:finish}.\\ 

We first start by proving the analogue of Proposition \ref{lem:corridor}, namely: For every $r>0$, there exist constants $L_0(r), N_0(r)$ such that the following holds: For every combinatorial geodesic  $\gamma$ between two points $x, y$ of $X$, for every sub-segment $\gamma_- \subset \gamma$ of length at least $L_0(r)$, there exist at most $N_0(r)$ elements $h$ of $G$ such that $d(x,hx), d(y,hy)<r$, and such that 
 there exist a quadrangle between $\gamma$ and $h \gamma$ 
and a sub-quadrangle between $\gamma_-$ and a sub-segment $\gamma_+ \subset h\gamma$ which is a corridor of width at most $4r+36C_\Box = 2(2r+16C_\Box) + 4C_\Box$.

The proof, for abritrary points of $X$ instead of points of $\Lambda_1$, is almost identical to the proof presented in Section \ref{sec:corridor}, up to modifying the constants according to the aforementioned changes (in particular, the constant $|g|_1$ is no longer needed in the definition of $L_0(r)$ and $N_0(r)$ ). The only notable difference is that Step 3 is no longer needed. Indeed, Steps 1-2 and 4-7 yield a geodesic segment of $X$ of length $L$ pointwise fixed by $N$ group elements, which is sufficient to contradict the weak acylindricity of the action (in Section \ref{sec:corridor}, Step 3 was there to show that such a geodesic segment is actually contained in a combinatorial axis of the WPD element $g$, which was necessary to contradict the weak WPD condition).\\

We then adapt the results of Section \ref{sec:finish} for abritrary points of $X$ instead of points of $\Lambda_1$, namely, we prove the following: For every $r>0$, there exist constants $L(r), N(r)$ such that the following holds: For every two points $x, y$ of $X$ at distance at least $L(r)$, there exist at most $N(r)$ elements $h$ of $G$ such that $d(x,hx), d(y,hy)<r$.

The proof of Steps 1-7 is identical, and yield a geodesic segment of $X$ of length $L$ pointwise fixed by $N$ group elements, which is sufficient to contradict the weak acylindricity of the action (in Section \ref{sec:finish}, Steps 1-7 only dealt with points of the axis $\Lambda_1$ and Step 8 dealt with arbitrary points of $X$).

\section{The geometry of generalised Higman groups}\label{sec:Higman}

The \textit{generalised Higman groups} $H_n$, $n \geq 5$ were defined in \cite{HigmanGroup},  as the groups with the following presentation: 

$$H_n:= \langle a_i, i \in \bbZ / n \bbZ ~|~ a_ia_{i+1}a_i^{-1}= a_{i+1}^{2}, i \in \bbZ / n \bbZ  \rangle.$$

These groups can naturally be seen as negatively curved polygons of groups (see for instance \cite[Theorem 12.28 and Section 12.29]{BridsonHaefliger}). The local group associated to the polygon is trivial. Order cyclically the edges of the polygon into a sequence $(e_i)$. The local group associated to each edge $e_i$ of the polygon is a copy of $\bbZ$ with a chosen generator $a_i$. The local group associated to each vertex of the polygon is the Baumslag--Solitar group $BS(1,2)$.  Finally, for each vertex $v$ contained in edges $e_i$ and $e_{i+1}$, the local maps $G_{e_{i}}\ra G_v$, $G_{e_{i+1}}\ra G_v$ send the generators $a_i, a_{i+1}$ respectively to generators $b_i, b_{i+1}$ of $BS(1,2)$ satisfying the relation $b_ib_{i+1}b_i^{-1}= b_{i+1}^{2}$.

Thus, $H_n$, $n\geq5$ acts cocompactly on a CAT($-1$) polygonal complex with strict fundamental domain an $n$-gon. Stabilisers of faces are trivial, stabilisers of edges are infinite cyclic, and stabilisers of vertices are isomorphic to $BS(1,2)$. Moreover, for every edge $e$ of $X_n$ with vertices $v$ and $v'$, one of the morphisms $G_e \hra G_v$, $G_e \hra G_{v'}$ is conjugated to $\langle a \rangle \hra \langle a,b | aba^{-1}=b^2 \rangle$ while the other one is conjugated to  $\langle b \rangle \hra \langle a,b | aba^{-1}=b^2 \rangle$.

\begin{lem}\label{lem:Higman_weak_acylindricity}
 If $v$ and $v'$ are vertices of $X_n$ at distance at least $3$, then the stabilisers of $v$ and $v'$ intersect trivially. In particular, $H_n$ acts weakly acylindrically on $X_n$ for every $n \geq 5$.
\end{lem}

\begin{proof}
 Since $H_n$ acts on the CAT(0) polygonal complex $X_n$ with trivial face stabilisers, fixed point sets of subgroups of $H_n$ are trees in the $1$-skeleton of $X_n$. If an element $h$ of $H_n$ stabilises a sequence $e_1, e_2, e_3$ of three adjacent edges, then for some vertex $v$ of $e_2$, the inclusion $G_e \hra G_{v}$  is conjugated to $\langle a \rangle \hra \langle a,b | aba^{-1}=b^2 \rangle$. It then follows that for every edge $e' \neq e$ containing $v$, the intersection $G_e \cap G_{e'}$ is trivial. Indeed, this is exactly Lemma 2.1 of \cite{MartinHigmanCubical} in the case of $4$ generators; the proof being purely about links of vertices, it generalises without any change to $n \geq 5$ generators. Thus, the intersection $G_{e_1} \cap G_{e_2} \cap G_{e_3}$ is trivial.
\end{proof}

Notice that, for $n \geq 5$, the CAT(-1) polygonal complex $X_n$ can be naturally subdivided into a hyperbolic CAT(0) square complex. Indeed, the subgraph of the one-skeleton of the first barycentric subdivision of $X_n$, obtained by removing edges corresponding to the inclusion of a vertex of $X_n$ in a polygon of $X_n$, is naturally the one-skeleton of a CAT(0) square complex. Moreover, such a CAT(0) square complex is quasi-isometric to $X_n$, and thus hyperbolic. In particular, Theorem A implies the following:

\begin{cor}
 For $n \geq 5$, the action of $H_n$ on $X_n$ is acylindrical. \qed
\end{cor}

We now give two consequences of the acylindricity of the action, which do not follow automatically from the abstract acylidrical hyperbolicity of the group.

\begin{cor}[Strong Tits alternative for generalised Higman groups]
 For $n \geq 5$, the group $H_n$ satisfies the following form of Tits alternative: A non-cyclic subgroup of $H_n$ is either contained in a vertex stabiliser, hence embeds in $BS(1,2)$, or is acylindrically hyperbolic. \qed
\end{cor}

\begin{prop}
 For $n\geq 5$, $H_n$ is residually $F_2$-free, i.e. each element of $H_n$ survives in a quotient of $H_n$ which does not contain a non-abelian free subgroup.
\end{prop}

\begin{proof}
 Let $h$ be an element of $H_n$. We  construct by induction a sequence of quotients of $H_n$ 
 $$H_n \xrightarrow{\pi_1} Q_1 \xrightarrow{\pi_2} Q_2 \rightarrow \cdots$$
 such that at each stage, $Q_k$ acts acylindrically on a hyperbolic metric space $Y_k$, and $h$ projects to a non-trivial element of $Q_k$. Such a construction uses the theory of rotating families of Dahmani--Guirardel--Osin and follows the construction in the proof of \cite[Theorem 8.9]{DahmaniGuirardelOsin}; we refer the reader there for additional details. The group $Q_0:=H_n$ acts acylindrically on $Y_0:=X_n$. Let us order all the elements of $H_n$ that act hyperbolically on $X_n$ in a sequence $(h_k)_{k \geq 1}$. Let us assume that the quotients $Q_0, Q_1, \ldots Q_k$ and hyperbolic spaces $Y_0, \ldots, Y_k$ are defined. We look at the projection $\bar{h}_{k+1}$ of $h_k$ in $Q_k$. Since $Q_k$ acts acylindrically on the hyperbolic space $Y_k$, such an element acts either hyperbolically or elliptically on $Y_k$ by a result of Bowditch \cite[Lemma 2.2]{BowditchTightGeodesics}.
 
 If $h_{k+1}$ projects to a an element of $Q_k$ acting hyperbolically on $Y_k$, we choose an integer $\alpha_{k+1} \geq 1$ such that 
 \begin{itemize}
 \item the coned-off space associated to the action of $\ll \bar{h}_{k+1}^{\alpha_{k+1}}\gg$ on $Y_k$ is hyperbolic,
  \item the quotient $\rquotient{Q_k}{\ll\bar{h}_{k+1}^{\alpha_{k+1}}\gg}$ acts acylindrically on the aforementioned  coned-off space,
  \item the projection map $Q_k \ra \rquotient{Q_k}{\ll\bar{h}_{k+1}^{\alpha_{k+1}}\gg}$ embeds a chosen metric ball of $Q_k$ containing the projection $\bar{h}$ of $h$.
 \end{itemize}
  We then define $Q_{k+1}:=\rquotient{Q_k}{\ll \bar{h}_{k+1}^{\alpha_{k+1}}\gg}$ and $Y_{k+1}$ is the associated coned-off space.

  If $h_{k+1}$ projects to a torsion $\bar{h}_{k+1}$ element of $Q_k$, we set $Q_{k+1}:= Q_{k}$, $Y_{k+1}:= Y_k$, and $\alpha_{k+1}\geq 1$ be an integer such that $\bar{h}_{k+1}^{\alpha_{k+1}}$ is trivial. 
  
 It thus follows that the quotient 
 $$ Q:= \rquotient{H_n}{\ll h_k^{\alpha_k}, k \geq 1\gg}$$
 of $H_n$ is such that $h$ projects to a non-trivial element. 
 
 Let us prove by contradiction that $Q$ does not contain a non-abelian free subgroup. If $Q$ contained such a subgroup, the preimage of such a subgroup under the projection $H_n \ra Q$ would be a free subgroup $F$ of $H_n$. As elements acting hyperbolically on $X_n$ are mapped to torsion elements of $Q$ by construction, we will derive a contradiction if we can prove that $F$ necessarily contains an element acting hyperbolically.\\
 
 \noindent \textbf{Claim:} The subgroup $F$ contains an element acting hyperbolically on $X_n$. \\
 
 If that was not the case, then every element of $F$ acts elliptically on $X_n$ (since $X_n$ has only finitely many isometry types of cells \cite{BridsonSemiSimple}). If the fixed point sets of elements of $F$ pairwise intersect, they globally intersect by a version of the Helly Theorem \cite[Proposition 5.3]{KleinerHellyCAT(0)}, hence $F$ is contained in a vertex stabiliser, which is absurd as $BS(1,2)$ does not contain non-abelian free subgroups. Thus, there must exist two elements $a,b$ of $F$ with disjoint fixed point sets. One could adapt the proof of \cite[Proposition 3.2]{MartinHigmanCubical} to show that this implies that there exists an element acting hyperbolically on $X_n$ in the subgroup generated by $a$ and $b$. In order to remain as self-contained as possible, we give a direct proof that the subgroup $\langle a, b\rangle$ contains an element acting hyperbolically, using the CAT(0) geometry of the space. 
 
 By Lemma \ref{lem:Higman_weak_acylindricity}, the fixed point sets of $a$ and $b$ are subcomplexes of diameter at most $2$. Let $x_a$, $x_b$ be points of the fixed point sets of $\mbox{Fix}(a)$, $\mbox{Fix}(b)$ which realise the distance $\alpha>0$ between these fixed point sets. Let $L$ be the CAT(0) geodesic between $v_a$ and $v_b$. 
 
 We now show that the angle at $x_a$ (respectively $x_b$) between $L$ and $aL$ (respectively $L$ and $bL$) is at least $\pi$. Let $\sigma$ be the (unique) minimal face of $X$ containing the `germ of $L$' at $x_a$, that is, the minimal face of $X$ such that sufficiently small neighbourhoods of $x_a$ in $L$ are contained in $\sigma$. If $\sigma$ is an edge (and thus, $x_a$ is a vertex), then the result follows immediately: Indeed, for every edge $e$ of $X$ and $v$ one of its vertices, an element in the stabiliser of $v$ either fixes $e$ or sends it to an edge making an angle of at least $\pi$ with $e$, by definition of the action. So let us assume that $\sigma$ is a square. If $x_a$ is in the interior of an edge of $\sigma$, then $L$ must be perpendicular to that edge since $L$ realises the distance between the fixed point sets, and the result follows immediately. If $x_a$ is a vertex of $\sigma$, then no edge of $\sigma$ is in $\mbox{Fix}(a)$ since $L$ realises the distance between the two fixed point sets. In particular, for the two edges $e, e'$ of $\sigma$ containing $x_a$, the angle at $x_a$ between $e$ and $ae$ (respectively $e'$ and $ae'$) is at least $\pi$ by the previous remark, whence the angle at $x_a$ between $L$ and $aL$ is at least $\pi$.
 
 As the angle at $x_a$ (respectively $x_b$) between $L$ and $aL$ (respectively $L$ and $bL$) is at least $\pi$, the paths $L \cup aL$ and $L \cup bL$ are geodesics. In particular, the distance between $\mbox{Fix}(a)$ and $ba\mbox{Fix}(a)$ is at least $2\alpha-2$. Reasoning analogously, one shows that, for every $k \geq 2$,  the distance between $\mbox{Fix}(a)$ and $(ba)^k\mbox{Fix}(a)$ is at least $k\alpha-2$. In particular, $ba$ does not have bounded orbits, so it does not act elliptically, a contradiction. This proves the claim, and finishes the proof. 
\end{proof}

\bibliographystyle{plain}
\bibliography{Higman_Cubical}

\def\cprime{$'$} \def\cprime{$'$}
\begin{thebibliography}{10}

\bibitem{BestvinaBrombergFujiwara}
M.~Bestvina, K.~Bromberg, and K.~Fujiwara.
\newblock Constructing group actions on quasi-trees and applications to mapping
  class groups.
\newblock {\em Publ. {M}ath. {I}nst. {H}autes {\'E}tudes Sci.}, in press, 2014.

\bibitem{BestvinaFujiwaraMappingClassGroups}
M.~Bestvina and K.~Fujiwara.
\newblock Bounded cohomology of subgroups of mapping class groups.
\newblock {\em Geom. Topol.}, 6:69--89 (electronic), 2002.

\bibitem{BestvinaFujiwaraHigherRank}
M.~Bestvina and K.~Fujiwara.
\newblock A characterization of higher rank symmetric spaces via bounded
  cohomology.
\newblock {\em Geom. Funct. Anal.}, 19(1):11--40, 2009.

\bibitem{BowditchTightGeodesics}
B.~H. Bowditch.
\newblock Tight geodesics in the curve complex.
\newblock {\em Invent. Math.}, 171(2):281--300, 2008,
  doi:10.1007/s00222-007-0081-y.

\bibitem{BridsonSemiSimple}
M.~R. Bridson.
\newblock On the semisimplicity of polyhedral isometries.
\newblock {\em Proc. Amer. Math. Soc.}, 127(7):2143--2146, 1999.

\bibitem{BridsonHaefliger}
M.~R. Bridson and A.~Haefliger.
\newblock {\em Metric spaces of non-positive curvature}, volume 319 of {\em
  Grundlehren der Mathematischen Wissenschaften [Fundamental Principles of
  Mathematical Sciences]}.
\newblock Springer-Verlag, Berlin, 1999.

\bibitem{PropertyACAT(0)CubeComplexes}
J.~Brodzki, S.~J. Campbell, E.~Guentner, G.~A. Niblo, and N.~J. Wright.
\newblock Property {A} and {$\rm CAT(0)$} cube complexes.
\newblock {\em J. Funct. Anal.}, 256(5):1408--1431, 2009.

\bibitem{BuxWitzelLocallyConvex}
K.-U. Bux and S.~Witzel.
\newblock Local convexity in {CAT}($\kappa$) spaces.
\newblock ar{X}iv:1211.1871v2, 2014.

\bibitem{CantatLamyCremonaNormalSubgroups}
S.~Cantat and S.~Lamy.
\newblock Normal subgroups in the {C}remona group.
\newblock {\em Acta Math.}, 210(1):31--94, 2013.
\newblock With an appendix by Yves de Cornulier.

\bibitem{CapraceHumeAcylindrical}
P.-E. Caprace and D.~Hume.
\newblock Orthogonal forms of {K}ac-{M}oody groups are acylindrically
  hyperbolic.
\newblock {\em Ann. Inst. Fourier}, in press, 2015.

\bibitem{CapraceSageevRankRigidity}
P.-E. Caprace and M.~Sageev.
\newblock Rank rigidity for {CAT}(0) cube complexes.
\newblock {\em Geom. Funct. Anal.}, 21(4):851--891, 2011.

\bibitem{CoornaertDelzantPapadopoulos}
M.~Coornaert, T.~Delzant, and A.~Papadopoulos.
\newblock {\em G\'eom\'etrie et th\'eorie des groupes}, volume 1441 of {\em
  Lecture Notes in Mathematics}.
\newblock Springer-Verlag, Berlin, 1990.
\newblock Les groupes hyperboliques de Gromov. [Gromov hyperbolic groups], With
  an English summary.

\bibitem{DahmaniGuirardelOsin}
F.~Dahmani, V.~Guirardel, and D.~Osin.
\newblock Hyperbolically embedded subgroups and rotating families in groups
  acting on hyperbolic spaces.
\newblock {\em Mem. Amer. Math. Soc.}, in press, 2012.

\bibitem{DelzantMonodromies}
T.~Delzant.
\newblock A finiteness property on monodromies of holomorphic families.
\newblock ar{X}iv:1402.4384v1, 2014.

\bibitem{GruberSistoAcylindricallyHyperbolic}
D.~Gruber and A.~Sisto.
\newblock Infinitely presented graphical small cancellation groups are
  acylindrically hyperbolic.
\newblock ar{X}iv:1408.4488v2, 2014.

\bibitem{HaglundSemiSimple}
F.~Haglund.
\newblock Isometries of {CAT}(0) cube complexes are semi-simple.
\newblock ar{X}iv:0705.3386v1, 2007.

\bibitem{HaglundWiseSpecial}
F.~Haglund and D.~T. Wise.
\newblock Special cube complexes.
\newblock {\em Geom. Funct. Anal.}, 17(5):1551--1620, 2008,
  doi:10.1007/s00039-007-0629-4.

\bibitem{HamenstadtWeakAcylindricity}
U.~Hamenst{\"a}dt.
\newblock Bounded cohomology and isometry groups of hyperbolic spaces.
\newblock {\em J. Eur. Math. Soc. (JEMS)}, 10(2):315--349, 2008.

\bibitem{HigmanGroup}
G.~Higman.
\newblock A finitely generated infinite simple group.
\newblock {\em J. London Math. Soc.}, 26:61--64, 1951.

\bibitem{KleinerHellyCAT(0)}
B.~Kleiner.
\newblock The local structure of length spaces with curvature bounded above.
\newblock {\em Math. Z.}, 231(3):409--456, 1999.

\bibitem{MartinHigmanCubical}
A.~Martin.
\newblock On the cubical geometry of {H}igman's group.
\newblock ar{X}iv:1506.02837v2, 2015.

\bibitem{McCammondWiseFansLadders}
J.~P. McCammond and D.~T. Wise.
\newblock Fans and ladders in small cancellation theory.
\newblock {\em Proc. London Math. Soc. (3)}, 84(3):599--644, 2002,
  doi:10.1112/S0024611502013424.

\bibitem{MinasyanOsinTrees}
A.~Minasyan and D.~Osin.
\newblock Acylindrically hyperbolic groups acting on trees.
\newblock {\em Math. Ann.}, in press, 2015.

\bibitem{OsinAcylindricallyHyperbolic}
D.~Osin.
\newblock Acylindrically hyperbolic groups.
\newblock {\em Trans. Amer. Math. Soc.}, in press, 2015.

\bibitem{SageevCubeComplex}
M.~Sageev.
\newblock Ends of group pairs and non-positively curved cube complexes.
\newblock {\em Proc. London Math. Soc. (3)}, 71(3):585--617, 1995,
  doi:10.1112/plms/s3-71.3.585.

\bibitem{SchuppHigmanSQ}
P.~E. Schupp.
\newblock Small cancellation theory over free products with amalgamation.
\newblock {\em Math. Ann.}, 193:255--264, 1971.

\bibitem{SelaAcylindrical}
Z.~Sela.
\newblock Acylindrical accessibility for groups.
\newblock {\em Invent. Math.}, 129(3):527--565, 1997,
  doi:10.1007/s002220050172.

\bibitem{SistoContractingElements}
A.~Sisto.
\newblock Contracting elements and random walks.
\newblock ar{X}iv:1112.2666v2, 2013.

\end{thebibliography}

\Address

\end{document}